\documentclass{amsart}

\usepackage{amsmath}
\usepackage{amsthm}
\usepackage{amssymb}
\usepackage{enumerate}
\usepackage{fullpage}
\usepackage[all]{xy}
\usepackage{hyperref}
\usepackage{wasysym}
\usepackage{url}
\usepackage{graphicx}

\usepackage{float}
\floatstyle{boxed}
\restylefloat{figure}

\theoremstyle{definition}
\newtheorem{thm}{Theorem}[section]
\newtheorem{definition}[thm]{Definition}
\newtheorem{lemma}[thm]{Lemma}

\newtheorem{cor}[thm]{Corollary}

\newtheorem{claim}[thm]{Claim}
\newtheorem{remark}[thm]{Remark}

\newtheorem{example}[thm]{Example}
\newtheorem*{embeddinglemmatemplate}{Embedding Lemma Template}
\newtheorem{fact}[thm]{Fact}

\newcommand\N{\mathbb{N}}
\newcommand\Z{\mathbb{Z}}
\newcommand\Q{\mathbb{Q}}
\newcommand\R{\mathbb{R}}
\DeclareMathOperator\id{id}
\newcommand\T{\mathbb{T}}
\newcommand\Tlog{\mathbb{T}_{\log}}

\DeclareMathOperator\SPAN{span}

\DeclareMathOperator\sign{sign}

\newcommand\Gammalog{\Gamma_{\text{log}}}
\DeclareMathOperator\image{image}

\author{Allen Gehret}
\title{The Asymptotic Couple of the Field of Logarithmic Transseries}
\email{agehret2@illinois.edu}
\address{Department of Mathematics, University of Illinois at Urbana-Champaign, Urbana, Illinois 61801}
\date{\today}

\begin{document}

\maketitle

\begin{abstract}
The derivation on the differential-valued field $\mathbb{T}_{\log}$ of logarithmic transseries induces on its value group $\Gamma_{\log}$ a certain map $\psi$. The structure $(\Gamma_{\log},\psi)$ is a divisible asymptotic couple. We prove that the theory $T_{\log} = {\rm Th}(\Gamma_{\log},\psi)$ admits elimination of quantifiers in a natural first-order language. All models $(\Gamma,\psi)$ of $T_{\log}$ have an important discrete subset $\Psi:=\psi(\Gamma\setminus\{0\})$. We give explicit descriptions of all definable functions on $\Psi$ and prove that $\Psi$ is stably embedded in $\Gamma$. 
\end{abstract}

\tableofcontents

\section{Introduction}
\label{introduction}

\medskip\noindent
The differential-valued field $\T_{\log}$
of logarithmic transseries is conjectured to have good model theoretic properties. As a partial result in this direction, and as a confidence building measure we prove here that at least its
{\em asymptotic couple\/} has a good model theory: quantifier elimination, and
stable embeddedness of a certain discrete part. We now describe
the relevant objects and results in more detail.

\medskip\noindent
Throughout, $m$ and $n$ range over $\N=\{0,1,2,\dots\}$. 
See \cite{Oleron} for a definition of the differential-valued
field $\T_{\log}$ of logarithmic transseries.
It is a field extension of $\R$ containing elements
$\ell_0, \ell_1, \ell_2,\ldots$, to be thought of as 
$x, \log x, \log \log x,\ldots$, and 
the elements of 
$\T_{\log}$ are formal series with
real coefficients and monomials $ \ell_0^{r_0}\ell_1^{r_1}\cdots \ell_n^{r_n}$ (with arbitrary real exponents $r_0,\dots, r_n)$.
For our purpose it 
is enough to know the following four things about $\T_{\log}$, its elements $\ell_n$, and these monomials: \begin{enumerate} 
\item These monomials are the elements of a subgroup $\frak{L}$ of the multiplicative group of $\T_{\log}$, and their products
are formed in the way suggested by their notation as power products. The elements of $\frak{L}$ are also
known as {\em logarithmic monomials}. For $m\le n$ we
have $\ell_m=\ell_0^{r_0}\cdots\ell_n^{r_n}$ where $r_i=0$ for all $i\ne m$ and $r_m=1$. 
\item The field $\T_{\log}$ is equipped with a (Krull) valuation $v$
that maps the group $\frak{L}$ isomorphically onto the (additively written) value group 
$v(\T_{\log}^\times)=\bigoplus_n \R e_n$, a vector space over $\R$ with basis $(e_n)$, with
$$v(\ell_0^{r_0}\ell_1^{r_1}\cdots \ell_n^{r_n})\ = -\ r_0e_0 - \cdots - r_ne_n,$$
and made into an ordered group by requiring for nonzero
$\sum_i r_i e_i$ that
$$\sum r_i e_i >0\ \Longleftrightarrow\ r_n>0\  \text{\ for the least $n$ such that $r_n\ne 0$}. $$
\item The field $\T_{\log}$ is equipped with a derivation such that $\ell_0'=1, \ell_1'=\ell_0^{-1},$ and in general $\ell_n^\dagger=\ell_0^{-1}\cdots\ell_n^{-1}$. Here $f^\dagger:= f'/f$ denotes the logarithmic derivative
of a nonzero element $f$ of a differential field, obeying the
useful identity $(fg)^\dagger=f^\dagger + g^\dagger$.  
In $\T_{\log}$,
$$  (\ell_0^{r_0}\ell_1^{r_1}\cdots \ell_n^{r_n})^\dagger\ =\ r_0\ell_0^{-1} +r_1\ell_0^{-1}\ell_1^{-1}+\cdots + r_n\ell_0^{-1}\cdots \ell_n^{-1}.$$
\item This derivation has the property that for nonzero 
$f\in \T_{\log}$ with $v(f)\ne 0$, the value $v(f')$, and thus
$v(f^\dagger)$, depends only on $v(f)$.
\end{enumerate}
Let $\Gamma_{\log}$ be the above ordered abelian group 
$\bigoplus_n \R e_n$. For an arbitrary ordered abelian group $\Gamma$ we set
$\Gamma^{\ne}:= \Gamma\setminus \{0\}$. 
By (4) the derivation of $\T_{\log}$ induces maps 
$$\gamma \mapsto \gamma' \ \text{ and }\ \gamma \mapsto \gamma^{\dagger}\ \colon\ \Gamma_{\log}^{\ne}\to \Gamma_{\log}$$
as follows: if $\gamma=v(f)\ne 0$ with
$f\in \T_{\log}^\times$, then $\gamma'=v(f')$ and
$\gamma^\dagger=v(f^\dagger)$. We have 
$\gamma'=\gamma + \gamma^\dagger$
for $\gamma\in \Gamma_{\log}^{\ne}$, and we follow Rosenlicht~\cite{differentialvaluationII}
in taking the function 
$$\psi: \Gamma_{\log}^{\ne} \to \Gamma_{\log}, \qquad \psi(\gamma):= \gamma^\dagger$$
as a new primitive, calling the pair $(\Gamma_{\log}, \psi)$ the \textbf{asymptotic
couple of $\T_{\log}$}. 

\medskip\noindent
More generally, an \textbf{asymptotic couple} is a pair
$(\Gamma, \psi)$ where $\Gamma$ is an ordered abelian group and
$\psi: \Gamma^{\ne} \to \Gamma$ satisfies for all $\alpha,\beta\in\Gamma^{\neq}$,
\begin{itemize}
\item[(AC1)] $\alpha+\beta\neq 0 \Longrightarrow \psi(\alpha+\beta)\geq \min(\psi(\alpha),\psi(\beta))$;
\item[(AC2)] $\psi(r\alpha) = \psi(\alpha)$ for all $r\in\Z^{\neq}$, in particular, $\psi(-\alpha) = \psi(\alpha)$;
\item[(AC3)] $\alpha>0 \Longrightarrow \alpha+\psi(\alpha)>\psi(\beta)$.
\end{itemize}
If in addition for all $\alpha,\beta\in\Gamma$,
\begin{itemize}
\item[(HC)] $0<\alpha\leq\beta\Rightarrow \psi(\alpha)\geq \psi(\beta)$,
\end{itemize}
then $(\Gamma,\psi)$ is said to be of \textbf{$H$-type}, or to be an \textbf{$H$-asymptotic couple}. 

\medskip\noindent
The notion of asymptotic couple is due to Rosenlicht~\cite{differentialvaluationII} who focused on the
case where $\Gamma$ has finite rank as an abelian group or is finite-dimensional as a vector space over $\Q$ or $\R$. The asymptotic couple
$(\Gamma_{\log}, \psi)$ is of $H$-type, infinite-dimensional as 
vector space over $\R$, and the ordered subset 
$\Psi:=\psi(\Gamma_{\log}^{\ne})$ of $\Gamma_{\log}$ is isomorphic to $(\N; <)$. 
We determine here the elementary (i.e., first-order) theory of
$(\Gamma_{\log}, \psi)$, provide a quantifier elimination result in a natural language, and show that the induced structure on the set $\Psi$
is just its structure as an ordered subset
of $\Gamma_{\log}$ (so $\Psi$ is stably embedded
in $(\Gamma_{\log}, \psi)$).

\medskip\noindent
This paper is in the spirit of~\cite{closedasymptoticcouples}, which proves a quantifier elimination result for so-called \emph{closed asymptotic couples}. That paper did for the asymptotic couple of the field $\T$ of logarithmic-exponential transseries what is done here for the asymptotic couple of $\T_{\log}$. For an explicit construction of $\T$, see~\cite{logarithmicexponentialseries}. The main difficulty in getting QE, compared to
\cite{closedasymptoticcouples}, was to find the right extra primitives, and to establish a 
new Embedding Lemma ~\ref{Zinmiddle}. Our choice of primitives here yields a {\em universal\/} theory with QE. This makes some things simpler than in \cite{closedasymptoticcouples}, and has various other
benefits, as we shall see.

\subsection{Conventions}
\label{conventions}
By  ``ordered set'' we mean ``totally ordered set''.

\medskip\noindent
Let $S$ be an ordered set. Below, the ordering on $S$ will be denoted by $\leq$, and a subset of $S$ is viewed as ordered by the induced ordering. We put $S_{\infty}:= S\cup\{\infty\}$, $\infty\not\in S$, with the ordering on $S$ extended to a (total) ordering on $S_{\infty}$ by $S<\infty$. Occasionally, we even take two distinct elements $-\infty,\infty\not\in S$, and extend the ordering on $S$ to an ordering on $S\cup\{-\infty,\infty\}$ by $-\infty<S<\infty$. Suppose that $B$ is a subset of $S$. We put $S^{>B}:=\{s\in S:s>b\text{ for every $b\in B$}\}$ and we denote $S^{>\{a\}}$ as just $S^{>a}$; similarly for $\geq, <,$ and $\leq$ instead of $>$. For $a,b\in S\cup\{-\infty,\infty\}$ and $B\subseteq S$ we put
\[
[a,b]_{B}:=\{x\in B: a\leq x\leq b\}.
\]
If $B=S$, then we usually write $[a,b]$ instead of $[a,b]_S$. A subset $A$ of $S$ is said to be a \textbf{cut} in $S$, or \textbf{downward closed} in $S$, if for all $a\in A$ and $s\in S$ we have $s<a\Rightarrow s\in A$. We say that an element $x$ of an ordered set extending $S$ \textbf{realizes} the cut $A$ if $A = S^{<x}$. We say that $S$ is a \textbf{successor set} if every element $x\in S$ has an \textbf{immediate successor} $y\in S$, that is, $x<y$ and for all $z\in S$, if $x<z$, then $y\leq z$. For example, $\N$ and $\Z$ with their usual ordering are successor sets.

\medskip\noindent
Suppose that $G$ is an ordered abelian group. Then we set $G^{\neq}:=G\setminus\{0\}$. Also, $G^{<}:= G^{<0}$; similarly for $\geq,\leq,$ and $>$ instead of $<$. We define $|g| := \max\{g,-g\}$ for $g\in G$. For $a\in G$, the \textbf{archimedean class} of $a$ is defined by
\[
[a] := \{g\in G: |a|\leq n|g|\text{ and }|g|\leq n|a|\text{ for some }n\geq 1\}.
\]
The archimedean classes partition $G$. Each archimedean class $[a]$ with $a\neq 0$ is the disjoint union of the two convex sets $[a]\cap G^{<}$ and $[a]\cap G^{>}$. We order the set $[G]:=\{[a]:a\in G\}$ of archimedean classes by
\[
[a]<[b] :\Longleftrightarrow n|a|<|b|\text{ for all }n\geq 1.
\]
We have $[0]<[a]$ for all $a\in G^{\neq}$, and
\[
[a]\leq[b] :\Longleftrightarrow |a|\leq n|b| \text{ for some } n\geq 1.
\]
We say that $G$ is \textbf{archimedean} if $[G^{\neq}]:=[G]\setminus\{[0]\}$ is a singleton.

%\medskip\noindent
%As usual, $\Z$ is the ring of integers, $\Q$ is the field of rational numbers, and $\R$ is the field of real numbers. 

\section{Abstract Asymptotic Couples}
\label{asymptoticcouples}

\noindent
In this section we recall the basic theory of asymptotic couples, as defined in Section~\ref{introduction}. We conclude the section with an important example $(\Gamma_{\log}^{\Q},\psi)$, which will turn out to be a prime model of our theory $T_{\log}$. This example is essentially the same as $(\Gamma_{\log},\psi)$, except with $\Q$ everywhere instead of $\R$.

\medskip\noindent
\emph{Let $(\Gamma,\psi)$ be an asymptotic couple} (not necessarily of $H$-type). By convention we extend $\psi$ to all of $\Gamma$ by setting $\psi(0):=\infty$. Then $\psi(\alpha+\beta)\geq \min(\psi(\alpha),\psi(\beta))$ holds for all $\alpha,\beta\in\Gamma$, and $\psi:\Gamma\to\Gamma_{\infty}$ is a (non-surjective) valuation on the abelian group $\Gamma$. In particular, the following is immediate:

\begin{fact}
\label{valuationfact}
If $\alpha,\beta\in\Gamma$ and $\psi(\alpha)<\psi(\beta)$, then $\psi(\alpha+\beta)=\psi(\alpha)$.
\end{fact}

\medskip\noindent
For $\alpha\in\Gamma^{\neq}$ we shall also use the following notation:
\[
\alpha^{\dagger}:= \psi(\alpha), \quad \alpha':=\alpha+\psi(\alpha).
\]
The following subsets of $\Gamma$ play special roles:
\[
(\Gamma^{\neq})' := \{\gamma':\gamma\in\Gamma^{\neq}\}, \quad (\Gamma^{>})' := \{\gamma':\gamma\in\Gamma^{>}\},
\]
\[
\Psi := \psi(\Gamma^{\neq}) = \{\gamma^{\dagger}:\gamma\in\Gamma^{\neq}\} = \{\gamma^{\dagger}:\gamma\in\Gamma^{>}\}.
\]

\noindent
For an arbitrary asymptotic couple $(\Gamma',\psi')$ we may occasionally refer to the set $\psi'((\Gamma')^{\neq})$ as ``the $\Psi$-set of $(\Gamma',\psi')$''.

\medskip\noindent
We say that an asymptotic couple $(\Gamma,\psi)$ has \textbf{asymptotic integration} if
\[
\Gamma = (\Gamma^{\neq})'.
\]
Note that by AC3 we have $\Psi<(\Gamma^{>})'$.

\medskip\noindent
The following is~\cite[Proposition 3.1]{liouville} and generalizes~\cite[Proposition 3.1]{closedasymptoticcouples}. We repeat the proof here.

\begin{lemma}
\label{gaplemma}There is at most one $\beta$ such that
\[
\Psi<\beta<(\Gamma^{>})'.
\]
If $\Psi$ has a largest element, there is no such $\beta$.
\end{lemma}
\begin{proof}
If $\Psi\leq\alpha<\beta<(\Gamma^{>})'$, then $\gamma:=\beta-\alpha>0$ gives
\[
\gamma^{\dagger}\leq \alpha = \beta-\gamma<\gamma'-\gamma = \gamma^{\dagger},
\]
a contradiction.
\end{proof}

\begin{definition}
If $(\Gamma,\psi)$ contains an element $\beta$ as in Lemma~\ref{gaplemma}, then we say that \textbf{$(\Gamma,\psi)$ has a gap} and that $\beta$ is the \textbf{gap}.
\end{definition}

\medskip\noindent
The existence of gaps is part of an important trichotomy for $H$-asymptotic couples:

\begin{lemma}
\label{trichotomy}Suppose $(\Gamma,\psi)$ is of $H$-type. Then $(\Gamma,\psi)$ has exactly one of the following three properties:
\begin{enumerate}[(i)]
\item $(\Gamma,\psi)$ has a gap;
\item $\Psi$ has a largest element;
\item $\Gamma = (\Gamma^{\neq})'$, that is, $(\Gamma,\psi)$ has asymptotic integration.
\end{enumerate}
Moreover, $\Gamma$ has at most one element outside $(\Gamma^{\neq})'$.
\end{lemma}
\begin{proof}
This follows from~\cite[Lemma 3.1, Proposition 3.1]{closedasymptoticcouples}. See also~\cite[Corollary 9.2.16]{adamtt}.
\end{proof}

\medskip\noindent
 Note that if $(\Gamma,\psi)$ is an $H$-asymptotic couple, then $\psi$ is constant on archimedean classes of $\Gamma$: for $\alpha,\beta\in\Gamma^{\neq}$ with $[\alpha] = [\beta]$ we have $\psi(\alpha) = \psi(\beta)$. The function $\id+\psi$ enjoys the following remarkable intermediate value property:

\begin{lemma}
\label{ivp}
Suppose $(\Gamma,\psi)$ is of $H$-type. Then the functions
\[
\gamma\mapsto \gamma':\Gamma^{>}\to\Gamma,\quad \gamma\mapsto\gamma':\Gamma^{<}\to\Gamma
\]
have the intermediate value property.
\end{lemma}
\begin{proof}
\cite[Lemma 2.2 and Property (3), p. 320]{closedasymptoticcouples}. See also~\cite[Lemma 9.2.14]{adamtt}.
\end{proof}

\medskip\noindent
It is very useful to think of $H$-asymptotic couples in terms of the following geography:
\[
\Psi<\text{possible gap}<(\Gamma^{>})'
\]

\medskip\noindent
Let $(\Gamma,\psi)$ and $(\Gamma_1,\psi_1)$ be asymptotic couples. An \textbf{embedding}
\[
h:(\Gamma,\psi)\to(\Gamma_1,\psi_1)
\]
is an embedding $h:\Gamma\to\Gamma_1$ of ordered abelian groups such that
\[
h(\psi(\gamma)) = \psi_1(h(\gamma))\text{ for $\gamma\in\Gamma^{\neq}$.}
\]
If $\Gamma\subseteq\Gamma_1$ and the inclusion $\Gamma\to\Gamma_1$ is an embedding $(\Gamma,\psi)\to(\Gamma_1,\psi_1)$, then we call $(\Gamma_1,\psi_1)$ an \textbf{extension} of $(\Gamma,\psi)$.

\begin{definition}
Call an asymptotic couple $(\Gamma,\psi)$ \textbf{divisible} if the abelian group $\Gamma$ is divisible. If $(\Gamma,\psi)$ is a divisible asymptotic couple, then we construe $\Gamma$ as a vector space over $\Q$ in the obvious way.
\end{definition}

\medskip\noindent
As a torsion-free abelian group, we will consider $\Gamma$ as a subgroup of the divisible abelian group $\Q\Gamma:=\Q\otimes_{\Z}\Gamma$ via the embedding $\gamma\mapsto 1\otimes \gamma$. We also equip $\Q\Gamma$ with the unique linear order that makes it into an ordered abelian group containing $\Gamma$ as an ordered subgroup. By~\cite[Proposition 2.3(2)]{liouville}, $\psi$ extends uniquely to a map $(\Q\Gamma)^{\neq}\to\Q\Gamma$, also denoted by $\psi$, such that $(\Q\Gamma,\psi)$ is an asymptotic couple. We say that $(\Q\Gamma,\psi)$ is the \textbf{divisible hull} of $(\Gamma,\psi)$. Note that $\psi((\Q\Gamma)^{\neq}) = \Psi$ and $[\Q\Gamma] = [\Gamma]$. If $\dim_{\Q}\Q\Gamma$ is finite, then $\Psi = \psi(\Gamma^{\neq})$ is a finite set. We summarize this as follows:

\begin{lemma}
\label{divisibleclosure}Let $(\Gamma,\psi)$ be an asymptotic couple. Then $(\Q\Gamma,\psi)$ is an extension of $(\Gamma,\psi)$ such that
\begin{enumerate}
\item $(\Q\Gamma,\psi)$ is divisible,
\item $\psi((\Q\Gamma)^{\neq}) = \Psi$,
\item if $i:(\Gamma,\psi)\to(\Gamma_1,\psi_1)$ is an embedding and $(\Gamma_1,\psi_1)$ is divisible, then $i$ extends to a unique embedding $j:(\Q\Gamma,\psi)\to(\Gamma_1,\psi_1)$, and
\item if $(\Gamma,\psi)$ is of $H$-type, then so is $(\Q\Gamma,\psi)$.
\end{enumerate}
\end{lemma}

\begin{remark}
In terms of the trichotomy of asymptotic couples, (2) from Lemma~\ref{divisibleclosure} says that if $\max\Psi$ exists in $(\Gamma,\psi)$, then this property is preserved when passing to the divisible hull. However, it is entirely possible that $(\Gamma,\psi)$ has asymptotic integration whereas $(\Q\Gamma,\psi)$ has a gap. For an example of this, see the remark after Corollary 2 in~\cite{someremarks}. We avoid this pathology in Section~\ref{TheoryT} by adding the unary function symbols $\delta_1,\delta_2, \delta_3,\ldots$ to our language to ensure divisibility. 
\end{remark}

\begin{example}
\label{example1}
In analogy with $(\Gammalog,\psi)$ defined in Section~\ref{introduction}, we now define $(\Gamma_{\log}^{\Q},\psi)$. Let the underlying abelian group be $\bigoplus_n\Q e_n$, a vector space over $\Q$ with basis $(e_n)$. We make $\Gamma_{\log}^{\Q}$ into an ordered group by requiring for nonzero $\sum_{i}r_ie_i$ that
\[
\sum r_ie_i>0\Longleftrightarrow\text{$r_n>0$ for the least $n$ such that $r_n\neq 0$.}
\]
It is often convenient to think of an element $\sum r_ie_i$ as the vector $(r_0,r_1,r_2,\ldots)$. Define $\psi:\Gamma_{\log}^{\Q,\neq}\to\Gamma_{\log}^{\Q}$ for nonzero $\alpha = (r_0,r_1,r_2,\ldots)$ as follows:
\begin{enumerate}
\item[(Step 1)] Take the unique $n$ such that $r_n\neq 0$ but $r_m = 0$ for $m<n$. Thus
\[
\alpha = (\underbrace{0,\ldots,0}_{n},\underbrace{r_n}_{\neq 0},r_{n+1},\ldots)
\]
\item[(Step 2)] Set $\psi(\alpha):= (\underbrace{1,\ldots,1}_{n+1},0,0,\ldots) = \sum_{k=0}^ne_k.$
\end{enumerate}
The reader should verify the following properties:
\begin{enumerate}
\item $(\Gamma_{\log}^{\Q},\psi)$ is a divisible $H$-asymptotic couple.
\item $(\Gamma_{\log}^{\Q},\psi)$ has asymptotic integration: for any $\alpha = (r_0,r_1,r_2,\ldots)$, take the unique $n$ such that $r_n\neq 1$ and $r_m=1$ for $m<n$. Thus
\[
\alpha = (\underbrace{1,\ldots,1}_n,\underbrace{r_n}_{\neq 1},r_{n+1},\ldots)
\]
and then $\beta:=(\underbrace{0,\ldots,0}_{n},r_n-1,r_{n+1},\ldots)$ is the unique element of $\Gamma_{\log}^{\Q}$ with $\beta' = \alpha$.
\item The set $\Psi = \psi(\Gamma_{\log}^{\Q,\neq})$ is a basis for $\Gamma_{\log}^{\Q}$ as a vector space over $\Q$.
\end{enumerate}
\end{example}

\section{Asymptotic Integration}
\label{asymptoticintegration}

\medskip\noindent
\emph{In this section, $(\Gamma,\psi)$ will be an $H$-asymptotic couple with asymptotic integration and $\alpha,\beta$ will range over $\Gamma$}. By Lemma~\ref{divisibleclosure} we may assume that $(\Gamma,\psi)$ is given as a substructure of some divisible $H$-asymptotic couple. Doing this allows us to multiply by $\frac{1}{n}$ in the proofs, for $n\geq 1$.

\begin{definition}
\label{successordefinition}
Given $\alpha$ we let $\int\alpha$ denote the unique $\beta\neq 0$ such that $\beta' = \alpha$ and we call $\beta=\int\alpha$ the \textbf{integral} of $\alpha$. This gives us a function $\int:\Gamma\to\Gamma^{\neq}$ which is the inverse of $\gamma\mapsto\gamma':\Gamma^{\neq}\to\Gamma$. We sometimes refer to the act of applying the function $\int$ as \textbf{integrating}. Note that $\int\alpha<0$ if $\alpha\in\Psi$.

\medskip\noindent
We define the \textbf{successor function} $s:\Gamma\to\Psi$ by $\alpha\mapsto \psi(\int\alpha)$. The successor function gets its name from the observation that in many cases of interest, such as the asymptotic couple of $\Tlog$, the ordered subset $\Psi$ of $\Gamma$ is a successor set, and for $\alpha\in\Psi$, the immediate successor of $\alpha$ in $\Psi$ is $s(\alpha)$. However in general, $\Psi$ as an ordered subset of $\Gamma$ is not a successor set; for example, if $(\Gamma,\psi)$ is a so-called \emph{closed asymptotic couple} considered in~\cite{closedasymptoticcouples}, then $\Psi$ is a dense ordered set and hence not a successor set.

\medskip\noindent
We also define the \textbf{contraction map} $\chi:\Gamma^{<}\to\Gamma^{<}$ by $\alpha\mapsto \int\psi(\alpha)$. The contraction map gets its name from the connection between asymptotic couples and contraction groups (for instance, see~\cite{kuhlmann1, kuhlmann2, someremarks}). We will only refer to $\chi$ in Section~\ref{embeddinglemmas}. Since $\chi$ can be defined in terms of $\psi$ and $\int$, and $\int$ can be defined in terms of $s$ as we will see in Lemma~\ref{integralidentity}, we choose to focus most of our attention on the function $s$.
 \end{definition}

\begin{lemma}[Integral Identity]
\label{integralidentity}
$\int\alpha = \alpha-s\alpha$.
\end{lemma}
\begin{proof}
Note that $(\int\alpha)' = \alpha$. Expanding this out gives $\psi(\int\alpha)+\int\alpha = s\alpha+\int\alpha = \alpha$.
\end{proof}

\medskip\noindent
The next lemma tells us, among other things, that for each $\alpha$, we get an increasing sequence:
\[
s\alpha<s^2\alpha<s^3\alpha<s^4\alpha<\cdots
\]
in $\Psi$.

\begin{lemma}
\label{successorincreasing}If $\alpha\in(\Gamma^{<})'$, then $\alpha<s(\alpha)$, and if $\alpha\in(\Gamma^{>})'$, then $\alpha>s(\alpha)$. In particular, if $\alpha\in\Psi$, then $\alpha<s(\alpha)$.
\end{lemma}
\begin{proof}
If $\alpha\in(\Gamma^{>})'$, then $\alpha>\psi(\int\alpha)$ by AC3. Thus assume that $\alpha\in(\Gamma^{<})'$ and let $\alpha = \beta'$ with $\beta<0$. Then
\begin{eqnarray*}
\alpha<s(\alpha) &\Leftrightarrow& \alpha<\psi(\textstyle\int\alpha) \\
&\Leftrightarrow& \alpha<\psi(\beta) \\
&\Leftrightarrow& \alpha-\psi(\beta)<0 
\end{eqnarray*}
and the latter is true since $\alpha-\psi(\beta) = \beta'-\psi(\beta) = \beta$.
\end{proof}

\medskip\noindent
By HC, if $[\alpha]>[\beta]$ then $\psi(\beta-\alpha) = \psi(\alpha)$. In the case where $[\alpha] = [\beta]$ and $\alpha$ and $\beta$ are both sufficiently far up the set $(\Gamma^{<})'$, the following lemma can be very useful:

\begin{lemma}[Successor Identity]
\label{succid}
If $s\alpha<s\beta$, then $\psi(\beta-\alpha) = s\alpha$.
\end{lemma}
\begin{proof}
Assume $s\alpha<s\beta$. We will prove that $[\beta-s\alpha]<[s\alpha-\alpha]$, and so $\psi(\beta-\alpha)=\psi(s\alpha-\alpha) = \psi(-\int \alpha) = s\alpha$. From $s\alpha<s\beta$ we get $\psi(\int \alpha)<\psi(\int \beta)$, which gives $[\int \beta]<[\int \alpha]$. First consider the case where $\alpha\in (\Gamma^{<})'$ and $s\alpha<\beta$. Then $\int \alpha<0$ and $s\alpha-\alpha>0$. Note that
\begin{eqnarray*}
[\beta-s\alpha]<[s\alpha-\alpha] &\Leftrightarrow& \beta-s\alpha < \tfrac{1}{n}(s\alpha-\alpha) \text{\;\;for all $n\geq 1$}\\
&\Leftrightarrow& \beta<s\alpha + \tfrac{1}{n}(s\alpha-\alpha) \text{\;\;for all $n\geq 1$} \\
&\Leftrightarrow& \beta<\psi (\textstyle\int \alpha) + \tfrac{1}{n} (-\textstyle\int \alpha) \text{\;\;for all $n\geq 1$} \\
&\Leftrightarrow& \beta< \psi(-\tfrac{1}{n}\textstyle\int \alpha)+ (-\tfrac{1}{n}\textstyle\int \alpha) \text{\;\;for all $n\geq 1$} \\
&\Leftrightarrow& \beta< (-\tfrac{1}{n}\textstyle\int \alpha)' \text{\;\;for all $n\geq 1$} \\
&\Leftrightarrow& \textstyle\int \beta< \frac{1}{n} (-\textstyle\int \alpha) \text{\;\;for all $n\geq 1$,}
\end{eqnarray*}
and the latter holds because $[\int \beta]<[\int \alpha]$. All other cases are similar.
\end{proof}

\medskip\noindent
It follows that $s$ can be defined in terms of $\psi$ if we allow a suitable ``external parameter'':

\begin{cor}
\label{spsigap}
Let $(\Gamma^*,\psi^*)$ be an $H$-asymptotic couple with asymptotic integration that extends $(\Gamma,\psi)$. Suppose $\gamma^*\in\Psi^*$ is such that $\Psi<\gamma^*$. Then $s(\alpha) = \psi^*(\alpha-\gamma^*)$ for all $\alpha\in\Gamma$.
\end{cor}

\medskip\noindent
Since $\Psi$ has no largest element, compactness yields an extension $(\Gamma^*,\psi^*)$ of $(\Gamma,\psi)$ with an element $\gamma^*$ as in Corollary~\ref{spsigap}. In Section~\ref{embeddinglemmas} we also give explicit constructions for extensions with this property in Lemma~\ref{Nontop} and Lemma~\ref{Zontop}. 

\medskip\noindent
Since $(\Gamma,\psi)$ has asymptotic integration, Corollary~\ref{trichotomy} tells us that $(\Gamma,\psi)$ most definitely does not have a gap. However, it is fun (also useful) to summarize Corollary~\ref{spsigap} with the following slogan:
\[
``s(x) = \psi(x - \text{gap that does not exist})"
\]
This fact is essential for Corollary~\ref{ssfunctions} and a variant of this device allows the proof of Lemma~\ref{Zinmiddle} to be carried out.
The following is immediate from Corollary~\ref{spsigap} and HC for $\psi$:

\begin{cor} The function $s$ has the following properties:
\begin{enumerate}
\item $s$ is increasing on $(\Gamma^{<})'$ and decreasing on $(\Gamma^{>})'$,
\item if $\alpha\in s(\Gamma)$, then $s^{-1}(\alpha)\cap(\Gamma^{>})'$ and $s^{-1}(\alpha)\cap(\Gamma^{<})'$ are convex in $\Gamma$,
\item if $s$ is injective on $\Psi$, then $s$ is strictly increasing on $\Psi$.
\end{enumerate}
\end{cor}

\medskip\noindent
The following lemma is also useful in understanding $s$ in terms of $\psi$.

\begin{lemma}[Fixed Point Identity]
\label{fixedpointidentity}
$\beta=\psi(\alpha-\beta)$ iff $\beta=s(\alpha)$.
\end{lemma}
\begin{proof}
Applying $\psi$ to $\int\alpha = \alpha-s\alpha$ gives $s\alpha = \psi(\alpha-s\alpha)$. Next, suppose that $\beta = \psi(\alpha-\beta)$. Then $\alpha = (\alpha-\beta)+\beta = (\alpha-\beta)+\psi(\alpha-\beta)$ and so $\int\alpha = \alpha-\beta$. Applying $\psi$ yields $s\alpha = \psi(\alpha-\beta) = \beta$.\end{proof}

\medskip\noindent
The following lemma is a more constructive version of~\cite[Lemma 4.6]{closedasymptoticcouples}, but will not be used in the rest of this paper.

\begin{lemma}[Limit Lemma]
\label{limitlemma}
Let $\alpha\in\Gamma$. Then $\gamma_0:=s^2\alpha\in\Psi$ and $\delta_0:=s^2\alpha-\int s\alpha\in(\Gamma^{>})'$ and the map
\[
\gamma\mapsto \psi(\gamma-\alpha):\Gamma\to\Gamma_{\infty}
\]
takes the constant value $s\alpha$ on the set $[\gamma_0,\delta_0]:=\{\gamma:\gamma_0\leq\gamma\leq\delta_0\}$.
\end{lemma}
\begin{proof}
Define $\beta_0:=-\int\psi\int\alpha = -\int s(\alpha)>0$. Then $\gamma_0 = \psi(\beta_0)=s^2(\alpha)\in\Psi$ and $\delta_0 = s^2\alpha-\int s\alpha = s^2\alpha+\beta_0=\psi(\beta_0)+\beta_0 = \beta_0'\in(\Gamma^{>})'$. First we calculate the values of $\psi(\gamma_0-\alpha)$ and $\psi(\delta_0-\alpha)$:
\begin{eqnarray*}
\psi(\gamma_0-\alpha) &=& \psi(s^2\alpha-\alpha) \\
&=& s\alpha \quad\text{(by Lemma~\ref{succid})} \\
\psi(\delta_0 - \alpha) &=& \psi(s^2\alpha-\textstyle\int s\alpha-\alpha) \\
&=& \psi((s^2\alpha - \alpha) - \textstyle\int s\alpha)\\
&=& s\alpha \quad\text{(because $\psi(s^2\alpha-\alpha) = s\alpha$)}
\end{eqnarray*}
Finally, we must show that $\psi(\gamma-\alpha)$ is constant as a function of $\gamma\in [\gamma_0,\delta_0]$. By HC, it is sufficient to show that either $\alpha<\gamma_0<\delta_0$ or $\gamma_0<\delta_0<\alpha$. First suppose $\alpha\in(\Gamma^{<})'$. By Lemma~\ref{successorincreasing} it follows that $\alpha<s\alpha<s^2\alpha = \gamma_0<\delta_0$. Otherwise suppose $\alpha\in(\Gamma^{>})'$. Then $\gamma_0<\delta_0$ and
\begin{eqnarray*}
\delta_0<\alpha &\Leftrightarrow& s^2\alpha-\textstyle\int s\alpha < \alpha \\
&\Leftrightarrow& -\textstyle\int s\alpha<\alpha-s^2\alpha.
\end{eqnarray*}
The inequality on the last line holds by HC and the observation that
\[
0<-\textstyle\int s\alpha<\alpha-s^2\alpha. \qedhere
\]
\end{proof}

\begin{lemma}
\label{s0lemma}
$s0\neq 0$ and $s0$ is the unique element $x\in\Gamma^{\neq}$ for which $\psi(x) = x$. 
\end{lemma}
\begin{proof}
By Lemma~\ref{successorincreasing} we have $s0\neq 0$, and by the Integral Identity $\int 0 = -s0$ and so $s0 = \psi(\int 0) = \psi(-s0) = \psi(s0)$. Uniqueness follows from the Fixed Point Identity: if $\psi(x) = x$, then $x = \psi(0-x)$ and so $x = s0$.
%If $x\in\Gamma^{\neq}$ and $\psi(x)=x$, then $x=\psi(0-x)$ and by the Fixed Point Identity, $x=s0$. If $s0 = 0$, then $0 = \psi(0) = \infty$, a contradiction.
\end{proof}

\medskip\noindent
Lemma~\ref{s0lemma} tells us that $H$-asymptotic couples with asymptotic integration come in two flavors: those with $s0>0$ and those with $s0<0$. The asymptotic couples $(\Gamma_{\log},\psi)$ and $(\Gamma_{\log}^{\Q},\psi)$ are both of type ``$s0>0$''. In the literature, when $s0>0$, then the element $s0$ is often denoted by ``1'' and then $(\Gamma,\psi)$ is said to \textbf{have a $1$}. We will not use this notation since we have the function $s$ at our disposal and we already will be making use of the rational number $1\in\Q$.

\begin{example}
\label{example2}
We return once again to the asymptotic couple $(\Gamma_{\log}^{\Q},\psi)$ defined in Example~\ref{example1}. Property (2) in Example~\ref{example1} already gives us the definition for the function $\int:\Gamma^{\Q}_{\log}\to\Gamma_{\log}^{\Q,\neq}$. Using $s = \psi\circ\int$, we can compute $s\alpha$ for $\alpha = (r_0,r_1,r_2,\ldots)\in\Gamma^{\Q}_{\log}$. Take the unique $n$ such that $r_n\neq 1$ and $r_m = 1$ for $m<n$. Thus
\[
\alpha = (\underbrace{1,\ldots,1}_n,\underbrace{r_n}_{\neq 1},r_{n+1},\ldots)
\]
and then
\[
s\alpha = (\underbrace{1,\ldots,1}_{n+1},0,0,\ldots).
\]
In particular, note that for elements in $\Psi$, $s$ acts as follows:
\begin{eqnarray*}
s(1,0,0,0,0,\ldots) &=& (1,1,0,0,0,\ldots) \\
s(1,1,0,0,0,\ldots) &=& (1,1,1,0,0,\ldots) \\
s(1,1,1,0,0,\ldots) &=& (1,1,1,1,0,\ldots) \\
&\vdots& \\
s(\underbrace{1,\ldots,1}_n,0,0,\ldots) &=& (\underbrace{1,\ldots,1}_{n+1},0,0,\ldots) \\
\end{eqnarray*}
Note that $s0 = (1,0,0,\ldots) = \min\Psi$ and $s0>0$. It is clear that the function $\gamma\mapsto s\gamma:\Psi\to\Psi^{>s0}$ is a bijection and $(\Psi;<)$ is a successor set such that for $\alpha<\beta\in\Psi$ we have $s\alpha\leq\beta$.
\end{example}

\section{The Embedding Lemma Zoo}
\label{embeddinglemmas}

\noindent
\emph{In this section, $(\Gamma,\psi)$ and $(\Gamma_1,\psi_1)$ are divisible $H$-asymptotic couples}. 
We include here many embedding results of the following form:

\begin{embeddinglemmatemplate}
Suppose $(\Gamma,\psi)$ has property $P$. Then there is a divisible $H$-asymptotic couple $(\Gamma',\psi')$ extending $(\Gamma,\psi)$ such that:
\begin{enumerate}
\item $(\Gamma',\psi')$ has property $Q$;
\item if $i:(\Gamma,\psi)\to(\Gamma_1,\psi_1)$ is an embedding such that $(\Gamma_1,\psi_1)$ has property $Q$, then $i$ extends uniquely to an embedding $j:(\Gamma',\psi')\to(\Gamma_1,\psi_1)$.
\end{enumerate}
\end{embeddinglemmatemplate}

\noindent
More often than not, properties $P$ and $Q$ involve the trichotomy presented in Lemma~\ref{trichotomy}. Recall that Lemma~\ref{trichotomy} states that $(\Gamma,\psi)$ has exactly one of the following properties:
\begin{itemize}
\item $(\Gamma,\psi)$ has a gap (``$\exists$ gap'');
\item $\Psi$ has a largest element (``$\exists\max\Psi$'');
\item $\Gamma = (\Gamma^{\neq})'$, that is, $(\Gamma,\psi)$ has asymptotic integration (``Asymptotic Integration'').
\end{itemize}

\noindent
In light of this, the author thought it would be helpful to the reader to include Figure~\ref{embeddinglemmasfigure} %1 % Hard-Coded Reference!!!!!!
 as a roadmap for navigating the various embedding results in terms of the trichotomy of Lemma~\ref{trichotomy}.
\begin{figure}[!htbp]
\centering
\caption{Embedding Lemmas for divisible $H$-asymptotic couples}
\label{embeddinglemmasfigure}
\resizebox{13cm}{!}{
\includegraphics{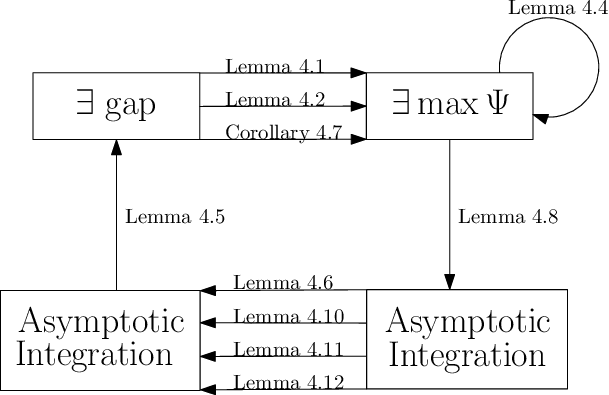}
}
\end{figure}

\medskip\noindent
The first two lemmas allow us to remove a gap by ``adjoining an integral'' for the gap. The first lemma shows that we can make the gap the derivative of a positive element; the lemma after that shows how to make the gap the derivative of a negative element.

\begin{lemma}[Removing a gap, positive version]
\label{gapupwards}
Let $\beta$ be a gap in $(\Gamma,\psi)$. Then there is a divisible $H$-asymptotic couple $(\Gamma+\Q\alpha,\psi^{\alpha})$ extending $(\Gamma,\psi)$ such that:
\begin{enumerate}
\item $\alpha>0$ and $\alpha' = \beta$;
\item if $i:(\Gamma,\psi)\to(\Gamma_1,\psi_1)$ is an embedding and $\alpha_1\in\Gamma_1$, $\alpha_1>0$, $\alpha_1'=i(\beta)$, then $i$ extends uniquely to an embedding $j:(\Gamma+\Q\alpha,\psi^{\alpha})\to(\Gamma_1,\psi_1)$ with $j(\alpha) = \alpha_1$.
\end{enumerate}
Furthermore, $\psi^{\alpha}((\Gamma+\Q\alpha)^{\neq}) = \Psi\cup\{\beta-\alpha\}$ with $\Psi<\beta-\alpha$.
\end{lemma}
\begin{proof}
This is similar to~\cite[Lemma 2.10]{liouville}. For the reader's convenience we mention that the ordering on $(\Gamma+\Q\alpha)$ is given by setting $0<q\alpha<\Gamma^{>}$ for all $q>0$ and $\psi^{\alpha}:(\Gamma+\Q\alpha)^{\neq}\to\Gamma+\Q\alpha$ is defined by
\[
\psi^{\alpha}(\gamma+r\alpha) :=
\begin{cases}
\psi(\gamma), &\text{if $\gamma\neq0$,} \\
\beta-\alpha, &\text{otherwise,}
\end{cases}
\]
for $\gamma\in\Gamma$ and $r\in\Q$, with $\gamma+r\alpha\neq 0$. See also~\cite[Lemma 9.8.2]{adamtt}.
\end{proof}

\begin{lemma}[Removing a gap, negative version]
\label{gapdownwards}
Let $\beta$ be a gap in $(\Gamma,\psi)$. Then there is a divisible $H$-asymptotic couple $(\Gamma+\Q\alpha,\psi^{\alpha})$ extending $(\Gamma,\psi)$ such that:
\begin{enumerate}
\item $\alpha<0$ and $\alpha' = \beta$;
\item if $i:(\Gamma,\psi)\to(\Gamma_1,\psi_1)$ is an embedding and $\alpha_1\in\Gamma_1$, $\alpha_1<0$, $\alpha_1'=i(\beta)$, then $i$ extends uniquely to an embedding $j:(\Gamma+\Q\alpha,\psi^{\alpha})\to(\Gamma_1,\psi_1)$ with $j(\alpha) = \alpha_1$.
\end{enumerate}
Furthermore, $\psi^{\alpha}((\Gamma+\Q\alpha)^{\neq}) = \Psi\cup\{\beta-\alpha\}$ with $\Psi<\beta-\alpha$.
\end{lemma}
\begin{proof}
This is similar to~\cite[Lemma 2.11]{liouville} and the construction of $(\Gamma+\Q\alpha,\psi^{\alpha})$ is similar to Lemma~\ref{gapupwards} except we set $\Gamma^{<}<q\alpha<0$ for all $q>0$.
\end{proof}

\begin{remark}
Lemmas~\ref{gapupwards} and~\ref{gapdownwards} show us that there are essentially two ways to remove a gap. These two ways are incompatible in the sense that given $(\Gamma,\psi)$ with gap $\beta$, we can obtain $(\Gamma^+,\psi^+)$ from Lemma~\ref{gapupwards} and $(\Gamma^-,\psi^-)$ from Lemma~\ref{gapdownwards} and there is no common extension $(\Gamma',\psi')$ in which these two can be amalgamated, i.e., the following configuration of embeddings is impossible:
\[
\xymatrix{
&(\Gamma',\psi')&\\
(\Gamma^+,\psi^+)\ar[ur]&&(\Gamma^-,\psi^-)\ar[ul]\\
&(\Gamma,\psi)\ar[ul] \ar[ur]&\\
}
\]
This issue is referred to as the ``fork in the road'' and is an obstruction to quantifier elimination. In~\cite{closedasymptoticcouples} this issue is resolved by adding an additional predicate to the language that ``decides'' for a gap whether it is supposed to be the derivative of a positive or of a negative element. We avoid this obstacle in Section~\ref{TheoryT} by adding the function $s$ to our language which ensures that all asymptotic couples considered already have asymptotic integration. The tradeoff in doing so is that we can only use embedding lemmas of the form
\[
\text{(Asymptotic Integration)} \to \text{(Asymptotic Integration)}
\]
(in the sense of Figure~1) in our proof of quantifier elimination.%%%HARD-CODED REFERENCE!!!!!!!!!!
\end{remark}

\medskip\noindent
If $(\Gamma,\psi)$ has a largest element $\beta$ in its $\Psi$-set, then Theorem~\ref{trichotomy} tells us that there is no $\alpha\in\Gamma$ such that $\alpha' = \beta$. Lemma~\ref{addnewlargestelement} tells us how to ``adjoin an integral'' for such an element $\beta$. It is important to note that the extension of $(\Gamma,\psi)$ constructed in Lemma~\ref{addnewlargestelement} also has a $\Psi$-set with a largest element.

\begin{lemma}[Adjoining an integral for $\max\Psi$]
\label{addnewlargestelement}
Assume $\Psi$ has a largest element $\beta$. Then there is a divisible $H$-asymptotic couple $(\Gamma+\Q\alpha,\psi^{\alpha})$ extending $(\Gamma,\psi)$ with $\alpha\neq0$, $\alpha'=\beta$, such that for any embedding $i:(\Gamma,\psi)\to(\Gamma_1,\psi_1)$ and any $\alpha_1\in\Gamma_1^{\neq}$ with $\alpha_1'=i(\beta)$ there is a unique extension of $i$ to an embedding $j:(\Gamma+\Q\alpha,\psi^{\alpha})\to(\Gamma_1,\psi_1)$ with $j(\alpha) = \alpha_1$. Furthermore, $\psi^{\alpha}((\Gamma+\Q\alpha)^{\neq}) = \Psi\cup\{\beta-\alpha\}$ with $\Psi<\beta-\alpha$.
\end{lemma}
\begin{proof}
This is a variant of~\cite[Lemma 2.12]{liouville}.
\end{proof}

\medskip\noindent
The next lemma allows us to add a gap to an asymptotic couple with asymptotic integration.

\begin{lemma}[Adding a gap]
\label{addgap}
Suppose $(\Gamma,\psi)$ has asymptotic integration. Then there is a divisible $H$-asymptotic couple $(\Gamma+\Q\beta,\psi_{\beta})$ extending $(\Gamma,\psi)$ such that:
\begin{enumerate}
\item $\Psi<\beta<(\Gamma^{>})'$;
\item for any $(\Gamma_1,\psi_1)$ extending $(\Gamma,\psi)$ and $\beta_1\in\Gamma_1$ with $\Psi<\beta_1<(\Gamma^{>})'$ there is a unique embedding $(\Gamma+\Q\beta,\psi_{\beta})\to(\Gamma_1,\psi_1)$ of asymptotic couples that is the identity on $\Gamma$ and sends $\beta$ to $\beta_1$;
\item the set $\Gamma$ is dense in the ordered abelian group $\Gamma+\Q\beta$, so $[\Gamma] = [\Gamma+\Q\beta]$, $\Psi = \psi_{\beta}((\Gamma+\Q\beta)^{\neq})$ and $\beta$ is a gap in $(\Gamma+\Q\beta,\psi_{\beta})$.
\end{enumerate}
\end{lemma}
\begin{proof}
This is~\cite[Lemma 9.8.4]{adamtt}. The proof uses a compactness argument.
\end{proof}

\medskip\noindent
Recall that a cut in an ordered set $S$ is simply a downward closed subset of $S$, and an element $a$ of an ordered set extending $S$ is said to realize the cut $C$ in $S$ if $C<a<S\setminus C$. The following Lemma~\ref{cutinclasses} is useful because it enables us to either:
\begin{enumerate}
\item add an element $\alpha$ witnessing $\psi(\alpha) = \beta$, if $\beta$ is not already in the $\Psi$-set, but is not disqualified from being in a larger $\Psi$-set by satisfying $\beta\in(\Gamma^{>})'$, or
\item add an additional archimedean class to $[\psi^{-1}(\beta)]$, if $\beta$ is already in the $\Psi$-set.
\end{enumerate}

\begin{lemma}
\label{cutinclasses}Let $C$ be a cut in $[\Gamma^{\neq}]$ and let $\beta\in \Gamma$ be such that $\beta<(\Gamma^{>})'$, $\gamma^{\dagger}\leq\beta$ for all $\gamma\in\Gamma^{\neq}$ with $[\gamma]\not\in C$, and $\beta\leq\delta^{\dagger}$ for all $\delta\in\Gamma^{\neq}$ with $[\delta]\in C$. Then there exists a divisible $H$-asymptotic couple $(\Gamma\oplus\Q\alpha,\psi^{\alpha})$ extending $(\Gamma,\psi)$, with $\alpha>0$, such that:
\begin{enumerate}
\item $[\alpha]\not\in[\Gamma^{\neq}]$ realizes the cut $C$ in $[\Gamma^{\neq}]$, $\psi^{\alpha}(\alpha) = \beta$;
\item given any embedding $i$ of $(\Gamma,\psi)$ into an $H$-asymptotic couple $(\Gamma_1,\psi_1)$ and any element $\alpha_1\in\Gamma_1^{>}$ such that $[\alpha_1]\not\in[i(\Gamma^{\neq})]$ realizes the cut $\{[i(\delta)]:[\delta]\in C\}$ in $[i(\Gamma^{\neq})]$ and $\psi_1(\alpha_1) = i(\beta)$, there is a unique extension of $i$ to an embedding $j:(\Gamma\oplus\Q\alpha,\psi^{\alpha})\to(\Gamma_1,\psi_1)$ with $j(\alpha) = \alpha_1$.
\end{enumerate}
If $(\Gamma,\psi)$ has asymptotic integration, then $(\Gamma^{\alpha},\psi^{\alpha})$ has asymptotic integration.
\end{lemma}
\begin{proof}
This is a variant of~\cite[Lemma 2.15]{liouville}.
\end{proof}

\medskip\noindent
For the special case of $C=\emptyset$ and $\beta$ a gap in $(\Gamma,\psi)$, Lemma~\ref{cutinclasses} gives:

\begin{cor}[Making the gap become $\max\Psi$]
\label{gaptomax}
Let $\beta\in\Gamma$ be a gap in $(\Gamma,\psi)$. Then there exists an $H$-asymptotic couple $(\Gamma+\Q\alpha,\psi^{\alpha})$ extending $(\Gamma,\psi)$, such that:
\begin{enumerate}
\item $0<q\alpha<\Gamma^{>}$ for all $q>0$, and $\psi^{\alpha}(\alpha)=\beta$;
\item for any embedding $i$ of $(\Gamma,\psi)$ into a divisible $H$-asymptotic couple $(\Gamma_1,\psi_1)$ and any $\alpha_1\in\Gamma_1^{>}$ with $\psi_1(\alpha_1) = i(\beta)$, there is a unique extension of $i$ to an embedding $j:(\Gamma+\Q\alpha,\psi^{\alpha})\to(\Gamma_1,\psi_1)$ with $j(\alpha) = \alpha_1$.
\end{enumerate}
\end{cor}

\medskip\noindent
Note that Corollary~\ref{gaptomax} is compatible with Lemma~\ref{gapdownwards} and incompatible with Lemma~\ref{gapupwards}: if $(\Gamma,\psi)$ has a gap $\beta$, then applying Corollary~\ref{gaptomax} ``decides'' that $\beta$ will be the derivative of a negative element in any extension with asymptotic integration.

\begin{lemma}[Divisible asymptotic integration closure]
\label{asymptoticintegrationclosure}
Let $(\Gamma_0,\psi_0)$ be a divisible $H$-asymptotic couple such that $\Psi$ has a largest element $\beta_0$. Then there exists a divisible $H$-asymptotic couple 
\[
(\Gamma,\psi) = (\Gamma_0\oplus\bigoplus_{n}\Q\alpha_{n+1},\psi) = (\Gamma_0\oplus\bigoplus_{n}\Q\beta_{n+1},\psi)
\]
 extending $(\Gamma_0,\psi)$ such that:
\begin{enumerate}
\item $(\Gamma,\psi)$ has asymptotic integration;
\item $s(\beta_n) = \beta_{n+1}$ and $\int\beta_n = \alpha_{n+1}$ for all $n$;
\item for any embedding $i$ of $(\Gamma_0,\psi_0)$ into a divisible $H$-asymptotic couple $(\Gamma^{\ast},\psi^{\ast})$ with asymptotic integration, there is a unique extension of $i$ to an embedding $(\Gamma,\psi)\to(\Gamma^{\ast},\psi^{\ast})$.
\end{enumerate}
\end{lemma}
\begin{proof}
For $n\geq 0$, define $(\Gamma_{n+1},\psi_{n+1})$ to be the asymptotic couple $(\Gamma_{n}+\Q\alpha_{n+1},\psi_{n}^{\alpha_{n+1}})$ constructed in Lemma~\ref{addnewlargestelement} as an extension of $(\Gamma_n,\psi_n)$. Set $\Psi_n:=\psi_n(\Gamma_n^{\neq})$ and note that $\Psi_{n+1} = \Psi_n\cup\{\beta_0-\sum_{k=0}^{n}\alpha_{k+1}\}$ with $\max\Psi_{n+1} = \beta_0-\sum_{k=0}^{n}\alpha_{k+1}=:\beta_{n+1}$. Let $(\Gamma,\psi) = \bigcup_n(\Gamma_n,\psi_n)$ and so $\Psi = \psi(\Gamma^{\neq}) = \bigcup_n\Psi_n$. Note that $\Psi$ does not have a maximum element. Furthermore, $(\Gamma,\psi)$ does not have a gap because it is the union of a chain of asymptotic couples which don't have gaps. Thus $(\Gamma,\psi)$ has asymptotic integration.

For (3), assume by induction that we have an embedding $i_n:(\Gamma_n,\psi_n)\to(\Gamma^{\ast},\psi^{\ast})$. Since $(\Gamma^{\ast},\psi^{\ast})$ has asymptotic integration, there is a unique extension of $i_n$ to an embedding $i_{n+1}:(\Gamma_{n+1},\psi_{n+1})$ such that $i_{n+1}(\alpha_{n+1}) = \int(i_{n}(\beta_n))$ by the universal property from Lemma~\ref{addnewlargestelement}. Thus there is a unique embedding $\cup_ni_n:(\Gamma,\psi)\to(\Gamma^{\ast},\psi^{\ast})$.
\end{proof}

\medskip\noindent
%The the extension of $(\Gamma_0,\psi_0)$ constructed in Lemma~\ref{asymptoticintegrationclosure} is unique up to isomorphism over $(\Gamma_0,\psi_0)$ by property (3). 
Given $(\Gamma_0,\psi_0)$ as in Lemma~\ref{asymptoticintegrationclosure}, the extension $(\Gamma,\psi)$ constructed in this lemma is the unique divisible $H$-asymptotic couple with asymptotic integration extending $(\Gamma_0,\psi_0)$ which has the universal property (3) in Lemma~\ref{asymptoticintegrationclosure}.
We call this extension \textbf{the divisible asymptotic integration closure} of $(\Gamma_0,\psi_0)$. The following summarizes the relationship between the $\alpha$'s and $\beta$'s in Lemma~\ref{asymptoticintegrationclosure}, with $\beta_0 = \max\Psi_0$. 

\[
\xymatrix{
\beta_0 \ar[rr]^-{s} \ar[ddrr]_-{\int}&& \beta_1 \ar[rr]^-{s} \ar[ddrr]_-{\int}&& \beta_2 \ar[rr]^-{s} \ar[ddrr]_-{\int}&&\beta_3 \ar[rr]^-{s} \ar[ddrr]_-{\int}&& \cdots \\
&&&&&&&&\\
&& \alpha_1 \ar[uu]^-{\psi} \ar[rr]_-{\chi}&& \alpha_2 \ar[uu]^-{\psi} \ar[rr]_-{\chi}&& \alpha_3 \ar[uu]^-{\psi} \ar[rr]_-{\chi}&&\cdots \\
}
\]
The diagram illustrates the manner in which we adjoined integrals at each stage of the construction.

\begin{example}
\label{example3}
Let $(\Gamma_0,\psi_0)\subseteq (\Gamma_{\log}^{\Q},\psi)$ be such that $\Gamma_0 = \Q e_0$. 
Then $e_0 = \max\psi(\Gamma_0^{\neq})$, and by the construction in Lemma~\ref{asymptoticintegrationclosure}, $(\Gamma_{\log}^{\Q}\psi)$ is the divisible asymptotic integration closure of $(\Gamma_0,\psi_0)$. Thus if $(\Gamma',\psi')$ is any divisible $H$-asymptotic couple with asymptotic integration such that $s0>0$, then there is an embedding
\[
i:(\Gamma_{\log}^{\Q},\psi)\to (\Gamma',\psi').
\]
%
%Since $e_0 = \max\psi(\Gamma_0^{\neq})$, follows from the construction in Lemma~\ref{asymptoticintegrationclosure} that $(\Gamma_{\log}^{\Q},\psi)$ is a divisible asymptotic integration closure of $(\Gamma_0,\psi_0)$. Next suppose that $(\Gamma',\psi')$ is an arbitrary divisible $H$-asymptotic couple with asymptotic integration such that $s0>0$. Then it follows from the universal property of Lemma~\ref{asymptoticintegrationclosure} that there is an embedding 
%\[
%i:(\Gamma_{\log}^{\Q},\psi)\to (\Gamma',\psi').
%\]
\end{example}

\begin{lemma}
\label{Nontop}Let $(\Gamma_0,\psi_0)$ be a divisible $H$-asymptotic couple with asymptotic integration. Then there exists a divisible $H$-asymptotic couple $(\Gamma,\psi) = (\Gamma_0\oplus\Q\alpha_0\oplus\bigoplus_{n}\Q\beta_0,\psi)$ extending $(\Gamma_0,\psi_0)$, such that:
\begin{enumerate}
\item $(\Gamma,\psi)$ has asymptotic integration;
\item $\psi_0(\Gamma_0^{\neq})<\beta_0<(\Gamma^{>}_0)'$, $\beta_0 = \psi(\alpha_0)$, $\beta_{n+1} = s(\beta_n)$ for all $n$;
\item for any embedding $i$ of $(\Gamma_0,\psi_0)$ into a divisible $H$-asymptotic couple $(\Gamma^*,\psi^*)$ with asymptotic integration and any $\alpha^*\in(\Gamma^*)^{<}$ such that $i(\psi_0(\Gamma_0))<\psi^*(\alpha^*)<(i(\Gamma_0)^>)'$, there is a unique extension of $i$ to an embedding $j:(\Gamma,\psi)\to (\Gamma^*,\psi^*)$ such that $j(\alpha_0) = \alpha^*$, $j(\beta_0) = \psi^*(\alpha^*)$ and $j(\beta_{k+1}) = s^k(\psi^{\ast}(\alpha^{\ast}))$.
\end{enumerate}
\end{lemma}
\begin{proof}
By Lemma~\ref{addgap}, we can extend $(\Gamma_0,\psi_0)$ to an asymptotic couple $(\Gamma_0\oplus\Q\beta_0,\psi)$ such that $\beta_0$ is a gap. Then by Corollary~\ref{gaptomax}, we can extend $(\Gamma_0\oplus\Q\beta_0,\psi)$ to an asymptotic couple $(\Gamma_0\oplus\Q\beta_0\oplus\Q\alpha_0,\psi)$ such that $\psi(\alpha_0) = \beta_0$. Thus $\beta_0 = \max\psi((\Gamma_0\oplus\Q\beta_0\oplus\Q\alpha_0)^{\neq})$.  Finally, we apply Lemma~\ref{asymptoticintegrationclosure} to this last asymptotic couple to obtain an asymptotic couple $(\Gamma,\psi) = (\Gamma_0\oplus\Q\beta_0\oplus\Q\alpha_0\oplus\bigoplus_{n}\Q\beta_{n+1},\psi)$ with the desired properties.
\end{proof}

\medskip\noindent
Setting $\alpha_{n+1}:=\int\beta_n = \chi\alpha_n$ in Lemma~\ref{Nontop}, we have the following configuration of the elements we adjoined to $(\Gamma_0,\psi_0)$:
\[
\xymatrix{
\beta_0 \ar[rr]^-{s} \ar[ddrr]_-{\int}&& \beta_1 \ar[rr]^-{s} \ar[ddrr]_-{\int}&& \beta_2 \ar[rr]^-{s} \ar[ddrr]_-{\int}&&\beta_3 \ar[rr]^-{s} \ar[ddrr]_-{\int}&& \cdots \\
&&&&&&&&\\
\alpha_0 \ar[uu]^-{\psi} \ar[rr]_-{\chi}&& \alpha_1 \ar[uu]^-{\psi} \ar[rr]_-{\chi}&& \alpha_2 \ar[uu]^-{\psi} \ar[rr]_-{\chi}&& \alpha_3 \ar[uu]^-{\psi} \ar[rr]_-{\chi}&&\cdots \\
}
\]
The top row of the above diagram is a ``copy of $\N$'' that has been added to the top of $\Psi_0$, i.e., $\Psi = \Psi_0\cup\{\beta_0,\beta_1,\ldots\}$ with $\Psi_0<\beta_0<\beta_1<\beta_2<\cdots$. The bottom row is a sequence of increasingly smaller and smaller elements (in the sense that $[\Gamma_0^{\neq}]>[\alpha_0]>[\alpha_1]>\cdots$) which serve as ``witnesses to the top row''.

\medskip\noindent
In the next lemma, we iterate the construction given by Lemma~\ref{Nontop} to add a ``copy of $\Z$'' to the top of the $\Psi$-set.

\begin{lemma}
\label{Zontop}Suppose $(\Gamma,\psi)$ has asymptotic integration and $\Gamma$ is divisible. Then there is a divisible $H$-asymptotic couple $(\Gamma_{\diamond},\psi_{\diamond})\supseteq(\Gamma,\psi)$ with a family $(\beta_k)_{k\in\Z}$ in $\Psi_{\diamond}$ such that:
\begin{enumerate}
\item $(\Gamma_{\diamond},\psi_{\diamond})$ has asymptotic integration;
\item $\Psi<\beta_0$, and $s(\beta_k) = \beta_{k+1}$ for all $k$;
\item for any embedding $i:(\Gamma,\psi)\to(\Gamma^*,\psi^*)$ into a divisible $H$-asymptotic couple with asymptotic integration and any family $(\beta_k^*)_{k\in\Z}$ in $\Psi^*$ with $i(\Psi)<\beta_0^*$, and $s(\beta_k^*) = \beta_{k+1}^*$ for all $k$, there is a unique extension of $i$ to an embedding $j:(\Gamma_{\diamond},\psi_{\diamond})\to(\Gamma^*,\psi^*)$ sending $\beta_k$ to $\beta_k^*$ for all $k$.
\end{enumerate}
\end{lemma}
\begin{proof}
For each $k\geq 0$, let $(\Gamma,\psi)\subseteq (\Gamma_k,\psi_k)$ be the extension given by Lemma~\ref{Nontop}. In the terms of the diagram below the proof of Lemma~\ref{Nontop}, label the sequence of $\beta$'s and $\alpha$'s in $(\Gamma_k,\psi_k)$ as $\beta_0^k,\beta_1^k,\ldots$ and $\alpha_0^k,\alpha_1^k,\ldots$. By the universal property of Lemma~\ref{Nontop}, there is a unique embedding $j_k:(\Gamma_k,\psi_k)\to(\Gamma_{k+1},\psi_{k+1})$ such that $\alpha_0^k\mapsto\alpha_1^{k+1}$.
\[
\xymatrix{
(\Gamma_k,\psi_k) \ar[r]^-{j_k} & (\Gamma_{k+1},\psi_{k+1}) \\
(\Gamma,\psi) \ar[u] \ar[ur]& \\
}
\]
This embedding results in identifications $\beta^k_l = \beta^{k+1}_{l+1}$ and $\alpha^k_l = \alpha^{k+1}_{l+1}$ for all $l\geq 0$. Thus we may define $(\Gamma_{\diamond},\psi_{\diamond})$ as the union of the increasing chain
\[
(\Gamma,\psi)\subseteq (\Gamma_0,\psi_0)\subseteq (\Gamma_1,\psi_1) \subseteq (\Gamma_2,\psi_2)\subseteq\cdots
\]
In $(\Gamma_{\diamond},\psi_{\diamond})$, we define $\beta_k:=\beta^0_k$ for $k\geq 0$ and $\beta_k:=\beta^{-k}_0$ for $k<0$. Furthermore we also define $\alpha_k:=\textstyle\int\beta_{k-1}$ for all $k$. The following table illustrates the identifications of the $\beta$'s in this increasing union, with elements in the same column being identified:
\[
\begin{array}{cccccccc}
\text{in }(\Gamma_{\diamond},\psi_{\diamond}): & \cdots & \beta_{-2} & \beta_{-1} & \beta_0 & \beta_1 & \beta_2 & \cdots \\
\vdots && \vdots & \vdots & \vdots & \vdots & \vdots & \\
\text{in }(\Gamma_2, \psi_2): && \beta_0^2 & \beta_1^2 & \beta_2^2 & \beta_3^2 & \beta_4^2 & \cdots \\
\text{in }(\Gamma_1,\psi_1): &&& \beta_0^1 & \beta_1^1 & \beta_2^1 & \beta_3^1 & \cdots \\
\text{in }(\Gamma_0,\psi_0): &&&& \beta_0^0 & \beta_1^0 & \beta_2^0 & \cdots \\
\end{array}
\]

The asymptotic couple $(\Gamma_{\diamond},\psi_{\diamond})$ has asymptotic integration since each $(\Gamma_k,\psi_k)$ has asymptotic integration. Furthermore, $\beta_0 = \beta_0^0>\Psi$ by Lemma~\ref{Nontop}. Also $s(\beta_l) = \beta_{l+1}$ for all $l\in\Z$. Indeed, if $l\geq 0$, then this is evident already in $(\Gamma_0,\psi_0)$. If $l<0$, then this can be observed in $(\Gamma_{-l},\psi_{-l})$ as $s(\beta_0^{-l}) = \beta_1^{-l}$.

Next, suppose $i:(\Gamma,\psi)\to(\Gamma^*,\psi^*)$ is an embedding into a divisible $H$-asymptotic couple with asymptotic integration and there is a family $(\beta_k^*)_{k\in\Z}$ in $\Psi^*$ with $i(\Psi)<\beta_0^*$ and $s(\beta_k^*) = \beta_{k+1}^*$ for all $k$. Since $si(\Psi)\subseteq i(\Psi)$, it follows that $i(\Psi)<\beta_k^*$ for all $k\in\Z$. Define the auxiliary $(\alpha^*_k)_{k\in\Z}$ in $\Gamma^*$ by $\alpha^*_k:= \int\beta^*_{k-1}$. Then we have $\psi(\alpha^*_k) = \beta^*_k$.

Next, for each $k\geq 0$, let $i_k:(\Gamma_k,\psi_k)\to (\Gamma^*,\psi^*)$ be the embedding given by Lemma~\ref{Nontop} with $i_k(\alpha_0^k) = \alpha^*_{-k}$.
\[
\xymatrix{
(\Gamma_k,\psi_k) \ar[r]^-{i_k}& (\Gamma^*,\psi^*) \\
(\Gamma,\psi) \ar[u] \ar[ur]_-{i}
}
\]
In order to show that this embedding extends to an embedding of $(\Gamma_{\diamond},\psi_{\diamond})$, we must show that $i_k\subseteq i_{k+1}$. It suffices to prove that $i_{k+1}(\beta_0^k) = i_k(\beta_0^k)$ and $i_{k+1}(\alpha_l^k) = i_k(\alpha_l^k)$ for all $l\geq 0$ which follows from a relatively straightforward diagram chase. Thus we get an embedding $j = \cup_ki_k:(\Gamma_{\diamond},\psi_{\diamond})\to(\Gamma^*,\psi^*)$.

\[
\xymatrix{
(\Gamma_{\diamond},\psi_{\diamond}) \ar[drr]^-{\cup_ki_k}\\
\vdots \ar[u] && (\Gamma^*,\psi^*) \\
(\Gamma_1,\psi_1) \ar[u] \ar[urr]^-{i_1}\\
(\Gamma_0,\psi_0) \ar[u] \ar[uurr]^-{i_0} \\
(\Gamma,\psi) \ar[u] \ar[uuurr]_-{i}\\
}
\]
It remains to prove uniqueness of $j$. Suppose $j':(\Gamma_{\diamond},\psi_{\diamond})\to(\Gamma^*,\psi^*)$ is an arbitrary embedding such that $j'(\beta_k) = \beta^*_k$ for all $k\in\Z$. It suffices to show that $j'|_{\Gamma_k} = i_k$ for all $k\geq 0$. I.e., $j'(\alpha_{-k}) = j'(\alpha_0^k) = i_k(\alpha_0^k) = \alpha^*_{-k}$. Integrating the expression $j'(\beta_{k-1}) = \beta_{k-1}^*$ yields
\[
\alpha_k^* = \textstyle\int j'(\beta_{k-1}) = j'(\textstyle\int\beta_{k-1}) = j'(\alpha_k). \qedhere
\]
\end{proof}

\medskip\noindent
The following lemma allows us to insert a ``copy of $\Z$'' into the middle or bottom of the $\Psi$-set of a divisible $H$-asymptotic couple with asymptotic integration in a canonical way. This \emph{was} the most tricky of the new embedding lemmas (\ref{asymptoticintegrationclosure},~\ref{Nontop},~\ref{Zontop},~\ref{Zinmiddle}) to establish; the decisive point in the proof is to pick an element $a^{\star}\in\Psi\setminus B$ and use it as in that proof.

\begin{lemma}
\label{Zinmiddle}Suppose $(\Gamma,\psi)$ is an $H$-asymptotic couple with asymptotic integration and $\Gamma$ is divisible. Let $B$ be a nonempty downward closed subset of $\Psi$ such that $s(B)\subseteq B$ and $B\neq \Psi$. Then there is a divisible $H$-asymptotic couple $(\Gamma_B,\psi_B)\supseteq (\Gamma,\psi)$ with a family $(\beta_k)_{k\in\Z}$ in $\Psi_B$ satisfying the following conditions:
\begin{enumerate}
\item $(\Gamma_B,\psi_B)$ has asymptotic integration;
\item $B<\beta_k<\Gamma^{>B}$, and $s(\beta_k) = \beta_{k+1}$ for all $k$;
\item for any embedding $i:(\Gamma,\psi)\to (\Gamma^*,\psi^*)$ into a divisible $H$-asymptotic couple with asymptotic integration and any family $(\beta_k^*)_{k\in\Z}$ in $\Psi^*$ such that $i(B)<\beta^*_k<i(\Gamma^{>B})$ and $s(\beta^*_k) = \beta^*_{k+1}$ for all $k$, there is a unique extension of $i$ to an embedding $(\Gamma_B,\psi_B)\to(\Gamma^*,\psi^*)$ sending $\beta_k$ to $\beta^*_k$ for all $k$.
\end{enumerate}
\end{lemma}
\begin{proof}
To motivate the construction of $(\Gamma_B,\psi_B)$ as required, suppose $(\Gamma_B,\psi_B)\supseteq(\Gamma,\psi)$ is an $H$-asymptotic couple with asymptotic integration and $(\beta_k)_{k\in\Z}$ a family in $\Psi_B$ such that $B<\beta_k<\Gamma^{>B}$ and $s(\beta_k) = \beta_{k+1}$ for all $k$.
Fix any $a^{\star}\in \Psi\setminus B$. Let $k\in\Z$ and note that by Corollary~\ref{succid},
\[
\psi_B(a^{\star}-\beta_k) = s(\beta_k) = \beta_{k+1}.
\]
Therefore setting $\alpha_k := \beta_{k-1}-a^{\star}$, we have $\psi_B(\alpha_k) = \beta_k$. Then $\alpha_k<0$ and
\[
[\gamma_1]<[\alpha_k]<[\alpha_{k+1}]<[\gamma_2]
\]
whenever $\gamma_1,\gamma_2\in\Gamma$ such that $\psi(\gamma_1)\in B$ and $\psi(\gamma_2)\in \Psi\setminus B$. Thus it follows that for $\gamma\in\Gamma$, $i_1<\cdots<i_n$ and $q_1,\ldots,q_n\in\Q$, we have
\[
\gamma+q_{1}\alpha_{i_1}+\cdots+q_n\alpha_{i_n}>_B0 \Longleftrightarrow \begin{cases} \gamma>0 &\text{if $\psi(\gamma)\in B$,} \\ \gamma>0 & \text{if $\psi(\gamma)\not\in B$ and $n=0$,}\\ q_1<0 & \text{if $\psi(\gamma)\not\in B$ and $n\geq 1$.} \\  \end{cases}
\]
and the $\psi_B$-value of such an element is uniquely determined:
\[
\psi_B(\gamma+q_{1}\alpha_{i_1}+\cdots+q_n\alpha_{i_n}) := \begin{cases} \psi(\gamma) &\text{if $\psi(\gamma)\in B$,} \\ \psi(\gamma) & \text{if $\psi(\gamma)\not\in B$ and $n=0$,}\\  \beta_{i_1}  & \text{if $\psi(\gamma)\not\in B$ and $n\geq 1$.} \\  \end{cases}
\]

Furthermore, note that $\alpha_k+\psi(\alpha_k) = \alpha_k+\beta_k = \alpha_k+a^{\star}+\alpha_{k+1}$. Rearranging terms gives us $\beta_{k-1} = a_k+a^{\star} = \alpha_k-\alpha_{k+1}+\psi(\alpha_k)$. Since $[\alpha_k]>[\alpha_{k+1}]$, it follows that $\psi(\alpha_k-\alpha_{k+1}) = \psi(\alpha_{k})$. Thus $\int\beta_{k-1} = \alpha_k-\alpha_{k+1}$, and so $s(\beta_{k-1}) = \psi(\alpha_k-\alpha_{k+1}) = \psi(\alpha_k) = \beta_k$. Furthermore, $\int\psi(\alpha_k-\alpha_{k+1}) = \int\psi(\alpha_k) = \alpha_{k+1}-\alpha_{k+2}$ implies that $\chi(\alpha_k-\alpha_{k+1}) = \chi(\alpha_k) = \alpha_{k+1}-\alpha_{k+2}$. Here is a picture of what is going on:

\[
\xymatrix{
\beta_{k-1} \ar[rrr]^-{s} \ar[rrdd]^-{\int}&&& \beta_{k} \ar[rrr]^-{s} \ar[rrdd]^-{\int}&&& \beta_{k+1} \\
\alpha_{k-1} \ar[u]_-{\psi} \ar[rrd]_-{\chi}&&& \alpha_{k} \ar[u]_-{\psi} \ar[rrd]_-{\chi} &&& \alpha_{k+1} \ar[u]_-{\psi}\\
&&  \alpha_k-\alpha_{k+1} \ar[uur]^-{\psi} \ar[rrr]^-{\chi} &&& \alpha_{k+1}-\alpha_{k+2} \ar[uur]^-{\psi} & \\
}
\]

Next, to actually obtain $(\Gamma_B,\psi_B)$, by compactness we take an elementary extension $(\Gamma_{\star},\psi_{\star})$ of $(\Gamma,\psi)$ with a family $(\beta_k)_{k\in\Z}$ in $\Psi_{\star}$ such that $B<\beta_k<\Gamma^{>B}$ and $s(\beta_k) = \beta_{k+1}$ for all $k$. Take $a^*\in\Psi\setminus B$ and define $\alpha_k := \beta_{k-1}-a^{\star}$. Set $\Gamma_B:= \Gamma+\sum_k\Q\alpha_k$. By the above observations, $(\Gamma_B,\psi_{\star}|_{\Gamma_B})$ is a divisible $H$-asymptotic couple with the desired properties.
\end{proof}

\medskip\noindent
Note that it follows from the proof of Lemma~\ref{Zinmiddle} that $\Psi_B = \Psi\cup\{\beta_k:k\in\Z\}$. 

\begin{lemma}
Let $i:\Gamma\to G$ be an embedding of divisible ordered abelian groups inducing a bijection $[\Gamma]\to[G]$. Then there is a unique function $\psi_G:G^{\neq}\to G$ such that $(G,\psi_G)$ is an $H$-asymptotic couple and $i:(\Gamma,\psi)\to (G,\psi_G)$ is an embedding.
\end{lemma}
\begin{proof}
\cite[Lemma 2.14]{liouville}.
\end{proof}

\begin{lemma}
\label{archclasslemma}Suppose $(\Gamma_0,\psi_0)\subseteq (\Gamma_1,\psi_1)$ and $(\Gamma^*,\psi^*)$ are divisible $H$-asymptotic couples, $i:(\Gamma_0,\psi_0 )\to(\Gamma^*,\psi^*)$ is an embedding and $j:\Gamma_1\to\Gamma^*$ is an ordered group embedding. Furthermore, suppose that $i = j|_{\Gamma_0}$ and $[\Gamma_0] = [\Gamma_1]$. Then $j$ is also an embedding of asymptotic couples, i.e., $j(\psi_1(\gamma)) = \psi^*(j(\gamma))$ for all $\gamma\in\Gamma_1^{\neq}$.
\end{lemma}
\begin{proof}
Let $\gamma\in\Gamma_1^{\neq}$. Since $[\Gamma_0] = [\Gamma_1]$, there is $\gamma_0\in\Gamma_0^{\neq}$ such that $[\gamma] = [\gamma_0]$. By HC $\psi_1(\gamma) = \psi_1(\gamma_0)$ and so $j(\psi_1(\gamma)) = j(\psi_1(\gamma_0)) = i(\psi_0(\gamma_0))$. Since $j$ is an ordered group embedding, it also follows that $[j(\gamma)] = [i(\gamma_0)]$ in $[\Gamma^*]$. Thus $\psi^*(j(\gamma)) = \psi^*(i(\gamma_0))$. Since $i$ is an embedding of asymptotic couples, $i(\psi_0(\gamma_0)) = \psi^*(i(\gamma_0))$ and we are done.
\end{proof}

\section{The Theory $T_{\log}$}
\label{TheoryT}

\noindent
Let $L_0$ be the ``natural'' language of asymptotic couples; $L_0 = \{0,+,-,<,\psi,\infty\}$ where $0,\infty$ are constant symbols, $+$ is a binary function symbol, $-$ and $\psi$ are unary function symbols and $<$ is a binary relation symbol. We consider an asymptotic couple $(\Gamma,\psi)$ as an $L_0$-structure with underlying set $\Gamma_{\infty}$ and the obvious interpretation of the symbols of $L_0$, with $\infty$ as a default value:
\[
-\infty = \gamma+\infty = \infty+\gamma=\infty+\infty = \psi(0) = \psi(\infty) = \infty
\]
for all $\gamma\in\Gamma$.

\medskip\noindent
Let $T_0$ be the $L_0$-theory whose models are the divisible $H$-asymptotic couples with asymptotic integration such that
\begin{itemize}
\item $\Psi$ as an ordered subset of $\Gamma$ has a least element $s0>0$,
\item $\Psi$ as an ordered subset of $\Gamma$ is a successor set and each $\alpha\in\Psi$ has immediate successor $s\alpha$, and
\item $\gamma\mapsto s\gamma:\Psi\to\Psi^{>s0}$ is a bijection.
\end{itemize}
It is clear that $(\Gamma_{\log},\psi)$ and $(\Gamma_{\log}^{\Q},\psi)$ are models of $T_0$. For a model $(\Gamma,\psi)$ of $T_0$, we define the function $p:\Psi^{>s0}\to\Psi$ to be the inverse to the function $\gamma\mapsto s\gamma:\Psi\to\Psi^{>s0}$. We extend $p$ to a function $\Gamma_{\infty}\to\Gamma_{\infty}$ by setting $p(\alpha) := \infty$ for $\alpha\in\Gamma_{\infty}\setminus\Psi^{>s0}$.

\medskip\noindent
Next let $L = L_0\cup\{s,p,\delta_1,\delta_2,\delta_3,\ldots\}$ where $s$, $p$ and $\delta_n$ for $n\geq 1$ are unary function symbols. All models of $T_0$ are considered as $L$-structures in the obvious way, again with $\infty$ as a default value, and with $\delta_n$ interpreted as division by $n$. 

\medskip\noindent
We let $T_{\log}$ be the $L$-theory whose models are the models of $T_0$. By adding function symbols $s,p,\delta_1,\delta_2,\ldots$ we have guaranteed the following:

\begin{lemma}
\label{universalaxiomatizationforT}
$T_{\log}$ has a universal axiomatization.
\end{lemma}

\medskip\noindent
Since $T_{\log}$ has a universal axiomatization, if $(\Gamma_1,\psi_1)\models T_{\log}$ and $(\Gamma_0,\psi_0)$ is an $L$-substructure of $(\Gamma_1,\psi_1)$, then $(\Gamma_0,\psi_0)\models T_{\log}$. This fact is very convenient for our proof of Quantifier Elimination in Theorem~\ref{qe} below. In Theorem~\ref{qe}, we actually prove an algebraic variant of Quantifier Elimination, which is indeed equivalent to Quantifier Elimination in the presence of a universal axiomatization. See~\cite[Prop 4.3.28]{marker} or~\cite[Prop B.11.14]{adamtt} for details of this equivalence.

\begin{thm}[Quantifier Elimination for $T_{\log}$]
\label{qe}
Suppose that $(\Gamma_0,\psi_0)\subsetneq (\Gamma_1,\psi_1)$ and $(\Gamma^*,\psi^*)$ are models of $T_{\log}$ such that $(\Gamma^*,\psi^*)$ is $|\Gamma_1|^+$-saturated, and $i:(\Gamma_0,\psi_0)\to(\Gamma^*,\psi^*)$ is an embedding of $L$-structures. Then there is an element $\alpha\in\Gamma_1\setminus\Gamma_0$ such that $i$ extends to an embedding $(\Gamma,\psi)\to(\Gamma^*,\psi^*)$ where $(\Gamma_0,\psi_0)\subseteq(\Gamma,\psi)\subseteq(\Gamma_1,\psi_1)$ and $\alpha\in\Gamma$.
\end{thm}
\begin{proof}
The general picture to keep in mind for this proof is the following:
\[
\xymatrix{
&(\Gamma^*,\psi^*) \\
(\Gamma_1,\psi_1) & \\
(\Gamma,\psi) \ar@{-}[u] \ar@{.>}[r]^-{\exists ?}_-{\sim}& (i\Gamma,\psi^{*}|_{i\Gamma}) \ar@{-}[uu] \\
(\Gamma_0,\psi_0) \ar[r]^-{i}_-{\sim} \ar@{-}[u]& (i\Gamma_0,\psi^*|_{i\Gamma_0}) \ar@{-}[u] \\
}
\]
Let $\Psi_1 := \psi_1(\Gamma_1^{\neq})$, $\Psi_0 :=\psi_0(\Gamma_0^{\neq})$ and $\Psi^*:=\psi^*((\Gamma^*)^{\neq})$. Note that $\Psi_0\subseteq\Psi_1$. The first two cases deal with the situation that $\Psi_0\neq\Psi_1$.

{\bf Case 1:} there is $\beta\in\Psi_1\setminus\Psi_0$ such that $\Psi_0<\beta$. Take such $\beta$, and define the family $(\beta_k)_{k\in\Z}$ by $\beta_0 := \beta$, $\beta_n := s^n\beta$ and $\beta_{-n} := p^n\beta$. Note that $s\beta_k = \beta_{k+1}$ for all $k\in\Z$. By Lemma~\ref{Zontop} we may assume that $(\Gamma_0,\psi_0)\subseteq (\Gamma_{\diamond},\psi_{\diamond})\subseteq (\Gamma_1,\psi_1)$ with $\beta\in\Gamma_{\diamond}$. By saturation of $(\Gamma^*,\psi^*)$, there is a family $(\beta^*_k)_{k\in\Z}$ in $\Gamma^*$ such that $i(\Psi_0)<\beta_0^*$ and $s(\beta_k^*) = \beta^*_{k+1}$ for all $k\in\Z$. Thus there is a unique extension of $i$ to an embedding $(\Gamma_{\diamond},\psi_{\diamond})\to(\Gamma^*,\psi^*)$ sending $\beta_k$ to $\beta_k^*$ for all $k\in\Z$.

{\bf Case 2:} $\Psi_1\neq\Psi_0$ and we are not in Case 1. Take $\beta_0\in\Psi_1\setminus\Psi_0$, and define the set $B:=\{\alpha\in\Psi_0:\alpha<\beta_0\}$. Note that $s(B) \subseteq B$. Also define the family $(\beta_k)_{k\in\Z}$ such that $\beta_n = s^n(\beta_0)$ and $\beta_{-n} = p^n(\beta_0)$ for $n>0$. Note that $B<\beta_k<\Gamma_0^{>B}$ and $s(\beta_k) = \beta_{k+1}$ for all $k\in\Z$. Thus by Lemma~\ref{Zinmiddle} we may assume that $(\Gamma_0,\psi_0)\subseteq (\Gamma_{0,B},\psi_{0,B})\subseteq (\Gamma_1,\psi_1)$. Again, by Lemma~\ref{Zinmiddle} and saturation of $(\Gamma^*,\psi^*)$, there is a family $(\beta_k^*)_{k\in\Z}$ in $\Gamma^*$ such that $i(B)<\beta_0^*<i(\Gamma_0^{>B})$ and $s\beta_k^* = \beta_{k+1}^*$, and so there is a unique extension of $i:(\Gamma_0,\psi_0)\to(\Gamma^*,\psi^*)$ to an embedding $(\Gamma_{0,B},\psi_{0,B})\to(\Gamma^*,\psi^*)$ that sends $\beta_k$ to $\beta_k^*$ for all $k\in\Z$.

{\bf Case 3:} $\Psi_0 = \Psi_1$ but $[\Gamma_0]\neq [\Gamma_1]$. Take some $\alpha\in\Gamma_1$ such that $[\alpha]\not\in[\Gamma_0]$. Let $\beta = \psi_1(\alpha)\in\Gamma_0$. Define $C$ to be the cut in $[\Gamma_0]$ which is realized by $[\alpha]$ in $[\Gamma_1]$. By Lemma~\ref{cutinclasses} there is an asymptotic couple $(\Gamma,\psi) = (\Gamma_0+\Q\alpha,\psi_1|_{\Gamma_0+\Q\alpha})$ extending $(\Gamma_0,\psi_0)$ inside $(\Gamma_1,\psi_1)$ and by saturation of $(\Gamma^*,\psi^*)$, the embedding $i$ extends to an embedding $(\Gamma,\psi)\to(\Gamma^*,\psi^*)$.

{\bf Case 4:} $[\Gamma_0] = [\Gamma_1]$ (and thus $\Psi_0 = \Psi_1$ by HC). By Quantifier Elimination for ordered divisible abelian groups, we get an extension $j:\Gamma_1\to\Gamma^*$ of $i$ as an embedding of ordered abelian groups:
\[
\xymatrix{
&&\Gamma^* \\
\Gamma_1\ar[urr]^{j}&& \\
&\Gamma_0 \ar[ul]\ar[uur]^{i}&\\
}
\]
By Lemma~\ref{archclasslemma}, $j$ is actually an embedding of asymptotic couples. Since $\Psi_0 = \Psi_1$ and $i$ is an embedding of $L$-structures, it follows that in fact $j$ also is an embedding of $L$-structures.
\end{proof}

\medskip\noindent
In the next corollary we collect the usual consequences of a quantifier elimination result:

\begin{cor}
$T_{\log}$ and $T_0$ are complete, decidable and model complete.
\end{cor}
\begin{proof}
Model completeness of $T_{\log}$ follows immediately from quantifier elimination, and completeness follows from the observation in Example~\ref{example3} that $(\Gamma_{\log}^{\Q},\psi)$ embeds into every model of $T_{\log}$.

For model completeness of $T_0$, assume that $(\Gamma,\psi)$ is a model of $T_0$. Then it suffices to show that the graphs of the functions $s,p,\delta_1,\delta_2,\ldots$ in $(\Gamma,\psi)$ are existentially definable. For the $\delta_n$'s we actually have quantifier-free definitions:
\[
\delta_n(\alpha) = \beta :\Longleftrightarrow n\beta = \alpha
\]
Furthermore, by the Fixed Point Identity, Lemma~\ref{fixedpointidentity}, we also get a quantifier-free definition for the function $s$:
\[
s(\alpha) = \beta :\Longleftrightarrow (\alpha\neq\infty \wedge \beta = \psi(\alpha - \beta)) \vee (\alpha = \beta = \infty)
\]
It remains to obtain an existential definition for the graph of the function $p$. We can define the graph of $p$ as follows:
\begin{eqnarray*}
p(\alpha) = \beta &:\Longleftrightarrow& (\beta\in\Psi\wedge s(\beta)=\alpha) \vee \\
&& (\beta=\infty \wedge \alpha\in\Gamma_{\infty}\setminus\Psi^{>s0} )\\
\end{eqnarray*}
If $\alpha\in \Gamma_{\infty}\setminus\Psi^{>s0}$, this can happen in one of four (not necessarily mutually exclusive) ways:
\begin{enumerate}
\item $\alpha=\infty$
\item $\alpha\in(\Gamma^{>})'$
\item $\alpha<s^2(0)$
\item there is a $\gamma\in\Gamma$ such that $\gamma^{\dagger}\neq \alpha$ and $s(\alpha) = s(\gamma^{\dagger})$.
\end{enumerate}
Thus we have the following existential definition for the graph of the function $p$:
\begin{eqnarray*}
p(\alpha) = \beta &:\Longleftrightarrow& \exists \gamma \;\text{such that}: \\
&& \quad(\gamma\neq 0 \wedge \beta=\gamma^{\dagger} \wedge s(\beta)=\alpha) \;\text{or} \\
&&\quad  (\beta=\infty\wedge \alpha=\infty)\;\text{or} \\
&&\quad (\beta=\infty\wedge\gamma>0\wedge\alpha = \gamma')\;\text{or} \\
&&\quad (\beta=\infty\wedge \alpha<s^2(0))\;\text{or} \\
&&\quad (\beta=\infty\wedge \gamma\neq0\wedge\gamma^{\dagger}\neq \alpha\wedge s(\alpha)=s(\gamma^{\dagger}))
\end{eqnarray*}

Completeness for $T_0$ follows from model completeness the same way it does for $T_{\log}$.

Decidability for both theories follows from completeness and the fact that these theories are formulated in recursive languages (a finite language in the case of $T_0$) and that they have recursively enumerable axiomatizations.
\end{proof}

\begin{example}
\label{examplesofqe}
Below are some quantifier free definitions of several definable sets in models of $T_{\log}$.
\begin{itemize}
\item The set $\Psi$ can be defined by the formula:
\[
x = p(s(x))
\]
\item The set $(\Psi-\Psi)^{>0}:=\{\alpha_1-\alpha_2:\alpha_1,\alpha_2\in\Psi\text{ and }\alpha_1>\alpha_2\}$ can be defined by the formula:
\[
x=-p(\psi(x)) + p(s(-(x-p(\psi(x))))) \wedge x\neq\infty
\]
It is left as an exercise to the reader to verify that this last formula does indeed define the set $(\Psi-\Psi)^{>0}$. (Hint: Use results from Section~\ref{definablefunctions}).
\end{itemize}
\end{example}

\section{Definable functions on $\Psi$}
\label{definablefunctions}

\noindent
For a (first-order) language $L$ and an $L$-structure $\mathcal{M}$ with underlying set $M$, we say that a set $D\subseteq M^m$ is \textbf{definable} if it is ``definable with parameters''. I.e., there is some $L$-formula $\varphi(x_1,\ldots,x_{m},y_1,\ldots,y_{n})$ and tuple $b = (b_1,\ldots,b_{n})\in M^n$ such that
\[
D = \{a\in M^m: \mathcal{M}\models \varphi(a,b)\}.
\]
Given $A\subseteq M$, a subset of $M^m$ is \textbf{definable over $A$} if the parameter $b$ above can be taken from $A^n$. A function $X\to M^n$ $(X\subseteq M^m)$ is definable (over $A$) if its graph is definable (over $A$).
%If $A\subseteq M$, we say that a set $D\subseteq M^m$ is \emph{definable over $A$} if the parameter $b$ above can be taken from $A^n$. For a set $X\subseteq M^m$, we say that a function $F:X\to M^n$ is definable (over $A$) if its graph is definable (over $A$). 

\medskip\noindent
\emph{In this section, we assume that $(\Gamma,\psi)\models T_{\log}$}.

\begin{definition}
For $k<0$ and $x\in\Psi$, we set $s^k(x) := p^{-k}(x)\in\Psi_{\infty}$. Also, $s^0(x):=x$ for all $x\in\Psi$. A function $F:\Psi\to\Gamma_{\infty}$ is an \textbf{$s$-function} if it is constant, or there are $n\geq 1, k_1<\cdots<k_n$ in $\Z$, $q_1,\ldots,q_n\in\Q^{\neq}$ and $\beta\in\Gamma$ such that $F(x) = \sum_{j=1}^nq_js^{k_j}(x)-\beta$ for all $x\in\Psi$.

For an $s$-function $F(x) = \sum_{j=1}^nq_js^{k_j}(x)-\beta$ as above with $n\geq 1$, define the set $D_F\subseteq\Psi$ to be
\[
D_F = \begin{cases}
[s^{-k_1+1}0,\infty)_{\Psi} & \text{if $k_1<0$} \\
\Psi & \text{if $k_1\geq 0$}
\end{cases}
\]
and the set $I_F\subseteq\Psi$ to be $\Psi\setminus D_F$.

Note that $I_F<D_F$, and  $I_F\cup D_F = \Psi$. Furthermore, $F$ takes the constant value $\infty$ on $I_F$ and takes only values in $\Gamma$ on $D_F$. It is also useful to note that for $x\in D_F$ and $l\in\Z$, if $l\geq k_1$, then $s^l(x)\in\Psi$.

By convention, if we refer to an $s$-function $F(x) = \sum_{j=1}^nq_js^{k_j}(x)-\beta$, it is understood that $n\geq 1$, $k_1<\cdots<k_n$ in $\Z$, $q_1,\ldots,q_n\in\Q^{\neq}$, and $\beta\in\Gamma$.
\end{definition}

\medskip\noindent
In general, the $s$-functions are rather well-behaved. To begin with, we get the following:

\begin{lemma}
\label{sfunctionsmonotone}
Let $F:\Psi\to\Gamma_{\infty}$ be the $s$-function $F(x) = \sum_{j=1}^n q_js^{k_j}(x)-\beta$. If $q_1>0$, then $F$ is strictly increasing on $D_F$, otherwise $F$ is strictly decreasing on $D_F$. In particular, the restriction of $F$ to $D_F$ is injective. Furthermore, if $F$ changes sign on $D_F$, then there is $\alpha\in D_F$ such that $\sign(F(\alpha))\neq \sign(F(s\alpha))$.
\end{lemma}
\begin{proof}
Let $\alpha_0,\alpha_1\in D_F$ be such that $\alpha_0<\alpha_1$. Then
\[
F(\alpha_1)-F(\alpha_0) = q_1(s^{k_1}\alpha_1-s^{k_1}\alpha_0)+q_2(s^{k_2}\alpha_1-s^{k_2}\alpha_0)+\cdots+q_n(s^{k_n}\alpha_1-s^{k_n}\alpha_0).
\]
By Lemma~\ref{succid}, we compute $\psi(s^{k_j}\alpha_1-s^{k_j}\alpha_0) = s^{k_j+1}\alpha_0$ for $j=1,\ldots,n$, and thus
\[
[s^{k_1}\alpha_1-s^{k_1}\alpha_0]>[s^{k_2}\alpha_1-s^{k_2}\alpha_0]>\cdots>[s^{k_n}\alpha_1-s^{k_n}\alpha_0].
\]
Since $s^{k_1}\alpha_1>s^{k_1}\alpha_0$, we get that
\[
\sign(F(\alpha_1)-F(\alpha_0)) = \sign(q_1).
\]
The second statement follows from an appeal to completeness of $T_{\log}$ and the observation that it is obviously true in $(\Gamma_{\log}^{\Q},\psi)$.
\end{proof}

\medskip\noindent
The following theorem is the main result of this section. It says that all definable functions $\Psi\to\Gamma_{\infty}$ are given piecewise by $s$-functions.

\begin{thm}
\label{psitogamma}
Let $F:\Psi\to\Gamma_{\infty}$ be a definable function. Then there is an increasing sequence $s0=\alpha_0<\alpha_1<\cdots<\alpha_{n-1}<\alpha_n=\infty$ in $\Psi_{\infty}$ such that for $k=0,\ldots,n-1$, the restriction of $F$ to $[\alpha_k,\alpha_{k+1})_{\Psi}$ is given by an $s$-function.
\end{thm}

\medskip\noindent
We first prove that for an $s$-function $F(x)$, the compositions $\psi(F(x))$, $s(F(x))$ are given piecewise by $s$-functions.

\medskip\noindent
The following lemma is a step in this direction.
%The calculation done in the proof of this lemma is very evident in the case $(\Gamma,\psi) = (\Gamma_{\log}^{\Q},\psi)$.

\begin{lemma}
\label{gencoeffsumlemma}
Let $n\geq 1$, $\alpha_1<\cdots<\alpha_n\in\Psi$, and let $\alpha = \sum_{j=1}^nq_j\alpha_j$ for $q_1,\ldots,q_n\in\Q^{\neq}$. Then
\begin{enumerate}[(1)]
\item $\sum_{j=1}^nq_j=0 \Longrightarrow\psi(\alpha) = s(\alpha_1)$,
\item $\sum_{j=1}^nq_j\neq 0\Longrightarrow\psi(\alpha) = s0$.
\end{enumerate}
\end{lemma}
\begin{proof} By completeness of $T_{\log}$, the lemma will follow from its validity for the case $(\Gamma,\psi) = (\Gamma_{\log}^{\Q},\psi)$. In $(\Gamma_{\log}^{\Q},\psi)$, we may take integers $0\leq m_1<\cdots<m_n$ such that $\alpha_j = \sum_{i=0}^{m_j}e_i$, for $j=1,\ldots,n$. Then
\[
\alpha = \sum_{j=1}^n q_j\left(\sum_{i=0}^{m_j}e_i\right) = \sum_{i=0}^{m_1}\left(\sum_{j=1}^nq_j\right)e_i + \sum_{i=m_1+1}^{m_2}\left(\sum_{j=2}^nq_j\right)e_i+\cdots+\sum_{i=m_{n-1}+1}^{m_n}q_ne_i,
\]
i.e., as an infinite tuple, $\alpha$ has the form:
\[
\alpha = (\underbrace{\textstyle\sum_{j=1}^nq_j,\ldots,\textstyle\sum_{j=1}^nq_j}_{m_1+1},\underbrace{\textstyle\sum_{j=2}^nq_j,\ldots,\textstyle\sum_{j=2}^nq_j}_{m_2-m_1},\ldots).
\]
From this it is clear that if $\sum_{j=1}^nq_j\neq 0$, then $\psi(\alpha) = e_0 = s0$. Otherwise, if $\sum_{j=1}^nq_j = 0$, then $q_1 = -\sum_{j=2}^nq_j\neq 0$ and so 
\[
\psi(\alpha) = \sum_{i=1}^{m_1+1}e_i = \alpha_1+e_{m_1+1} = s(\alpha_1). \qedhere
\]
\end{proof}

\medskip\noindent
The Fixed Point Identity (Lemma~\ref{fixedpointidentity}) which relates $\psi$ and $s$ immediately gives us an $s$-analogue of Lemma~\ref{gencoeffsumlemma}.

\begin{cor}
\label{sgencoeffsumcor}
Let $n\geq 1$, $\alpha_1<\cdots<\alpha_n\in\Psi$, and let $\alpha = \sum_{j=1}^nq_j\alpha_j$ for $q_1,\ldots,q_n\in\Q^{\neq}$. Then
\begin{enumerate}[(1)]
\item $\sum_{j=1}^nq_j\neq 1 \Longrightarrow s(\alpha) = s0$,
\item $\sum_{j=1}^nq_j=1 \Longrightarrow s(\alpha) = s(\alpha_1)$.
\end{enumerate}
\end{cor}
\begin{proof}
Suppose that $\sum_{j=1}^nq_j\neq 1$. Then $\sum_{j=1}^nq_j-1\neq 0$, so by Lemma~\ref{gencoeffsumlemma}, $\psi(\alpha-s0) = s0$. Thus $s\alpha = s0$ by Lemma~\ref{fixedpointidentity}.

Next, suppose that $\sum_{j=1}^nq_j=1$. Then $\sum_{j=1}^nq_j-1 = 0$, so by Lemma~\ref{gencoeffsumlemma}, $\psi(\alpha-s\alpha_1) = s\alpha_1$. Thus $s\alpha = s\alpha_1$ by Lemma~\ref{fixedpointidentity}.
\end{proof}

\medskip\noindent
In Theorem~\ref{psisfunctions} below we give an explicit description of how compositions $\psi(F(x))$ behave in all possible cases.
%
% In each of the cases, the proof is immediate from some combination of Fact~\ref{valuationfact}, Lemma~\ref{succid}, Lemma~\ref{gencoeffsumlemma} or Corollary~\ref{sgencoeffsumcor} or for some other simple reason. Instead of laboriously writing out all the details, we will just remark in the table for each case which of these get used. We will abbreviate Fact~\ref{valuationfact} as~\ref{valuationfact}, and so forth.

\begin{thm}
\label{psisfunctions}
Let $F:\Psi\to\Gamma_{\infty}$ be the $s$-function $F(x) = \sum_{j=1}^nq_js^{k_j}(x)-\beta$. Define the function $G:\Psi\to\Gamma_{\infty}$ by $G(x) = \psi(F(x))$. If $x\in I_F$, then $G(x) = \infty$. Otherwise, if $x\in D_F$, then the values $G(x)$ are given in the following table (with $q := \sum_{j=1}^nq_j$):
\begin{center}
\begin{tabular}{|c|l|}
\hline
$\beta$ & $G(x) = \psi(\sum_{j=1}^nq_js^{k_j}(x)-\beta)\quad\text{(assuming $x\in D_F$)}$ \\
%\hline
%if $\beta=0$ & $G(x) = 
%\begin{cases}
%s^{k_1+1}(x) & \text{if $q=0$ (\ref{gencoeffsumlemma})} \\
%s0 & \text{if $q\neq0$ (\ref{gencoeffsumlemma})} \\
%\end{cases}$ \\
\hline
if $\psi(\beta)>s0$ & $G(x) = 
\begin{cases}
s0 & \text{if $q\neq 0$} \\
s^{k_1+1}(x) & \text{if $q=0$ and $s^{k_1+1}(x)<\psi(\beta)$ } \\
G(s^{-k_1-1}(\psi(\beta))) & \text{if $q=0$ and $s^{k_1+1}(x) = \psi(\beta)$} \\
\psi(\beta) & \text{if $q=0$ and $s^{k_1+1}(x) > \psi(\beta)$}
\end{cases}$\\
\hline
if $\psi(\beta)=s0$ & $G(x) = 
\begin{cases}
G(s^{-k_1}s0) & \text{if $s^{k_1}(x) = s0$ } \\
s0 & \text{if $s^{k_1}(x)>s0$ and $q=0$ } \\
s^{k_1+1}(x) & \text{if $s^{k_1}(x)>s0$, $q\neq 0$, and $s^{k_1+1}(x)<s(q^{-1}\beta)$} \\
G(s^{-k_1-1}s(q^{-1}\beta)) & \text{if $s^{k_1}(x)>s0$, $q\neq 0$, and $s^{k_1+1}(x)=s(q^{-1}\beta)$} \\
s(q^{-1}\beta) & \text{if $s^{k_1}(x)>s0$, $q\neq 0$, and $s^{k_1+1}(x)>s(q^{-1}\beta)$} \\
\end{cases}$\\
\hline
\end{tabular}
\end{center}
\end{thm}
\begin{proof}
In the third, fifth and eighth cases in the table the computation is immediate since we are able to solve for $x$ in terms of $\beta$. For example, in the third case the assumption $s^{k_1+1}(x) = \psi(\beta)$ implies that $x = s^{-k_1-1}(\psi(\beta))$ and so the function takes the value $G(s^{-k_1-1}(\psi(\beta)))$.

Otherwise, the idea is to do a computation of the form $\psi(\alpha-\beta)$ where $\alpha = \sum_{j=1}^nq_js^{k_j}(x)$. In the first, second, fourth and fifth cases, we can compute the $\psi$-value of $\alpha$ by Lemma~\ref{gencoeffsumlemma} and the assumptions are such that the $\psi$-values of $\alpha$ and $\beta$ will be different so the $\psi$-value of their difference is immediate from Fact~\ref{valuationfact}.

For the seventh and ninth case, we have to compute
\[
\psi\big(\underbrace{q_1s^{k_1}(x)+\cdots+q_ns^{k_n}(x)}_{\alpha} - \beta\big)
\]
where by assumption $\psi(\beta) = \psi(\alpha) = s0$ since $q\neq 0$. Using (AC2), we can pivot to a situation where we can use Lemma~\ref{succid} and Corollary~\ref{sgencoeffsumcor} to do the computation. I.e., by dividing by $q$ we reduce to computing
\[
\psi\big(\underbrace{q^{-1}(q_1s^{k_1}(x)+\cdots+q_ns^{k_n}(x))}_{q^{-1}\alpha} - q^{-1}\beta\big).
\]
By Corollary~\ref{sgencoeffsumcor}, we know that $s(q^{-1}\alpha) = s^{k_1+1}(x)$. Our assumptions in cases seven and nine say precisely that the $s$-values of $q^{-1}\alpha$ and $q^{-1}\beta$ are different. From that point, it suffices to just use Lemma~\ref{succid}.
\end{proof}

\medskip\noindent
Corollary~\ref{spsigap} allows us to easily transform Theorem~\ref{psisfunctions} into an $s$-analogue. In the proof of Corollary~\ref{ssfunctions} below we perform this transformation.

\begin{cor}
\label{ssfunctions}
Let $F:\Psi\to\Gamma_{\infty}$ be the $s$-function $F(x) = \sum_{j=1}^nq_js^{k_j}(x)-\beta$. Define the function $G:\Psi\to\Gamma_{\infty}$ by $G(x) = s(F(x))$. If $x\in I_F$, then $G(x) = \infty$. Otherwise, if $x\in D_F$, then the values $G(x)$ are given in the following table:\begin{center}
\begin{tabular}{|c|l|}
\hline
$\beta$ & $G(x) = s(\sum_{j=1}^nq_js^{k_j}(x)-\beta)\quad\text{(assuming $x\in D_F$)}$ \\
\hline
if $s(-\beta)>s0$ & $G(x) = 
\begin{cases}
s0 & \text{if $q\neq 0$} \\
s^{k_1+1}(x) & \text{if $q=0$ and $s^{k_1+1}(x)<s(-\beta)$} \\
G(s^{-k_1-1}(s(-\beta))) & \text{if $q=0$ and $s^{k_1+1}(x) =s(-\beta)$} \\
s(-\beta) & \text{if $q=0$ and $s^{k_1+1}(x) > s(-\beta)$}
\end{cases}$\\
\hline
if $s(-\beta)=s0 $ & $G(x) = 
\begin{cases}
G(s^{-k_1}s0) & \text{if $s^{k_1}(x) = s0$} \\
s0 & \text{if $s^{k_1}(x)>s0$ and $q=0$} \\
s^{k_1+1}(x) & \text{if $s^{k_1}(x)>s0$, $q\neq 0$ and $s^{k_1+1}(x)<\gamma_0$} \\
G(s^{-k_1-1}\gamma_0) & \text{if $s^{k_1}(x)>s0$, $q\neq 0$ and $s^{k_1+1}(x)=\gamma_0$} \\
\gamma_0 & \text{if $s^{k_1}(x)>s0$, $q\neq 0$ and $s^{k_1+1}(x)>\gamma_0$} \\
\end{cases}$\\
\hline
%if $\beta=0$&$G(x) = 
%\begin{cases}
%s0&\text{if $q\neq 1$}\\
%s^{k_1+1}(x)&\text{if $q=1$}\\
%\end{cases}$\\
%\hline
\end{tabular}
\end{center}
 where $q := \sum_{j=1}^nq_j$ and for $q\neq 0$,
\[
\gamma_0 := \begin{cases}
s0 & \text{if $\beta=0,q\neq 1$} \\
\infty & \text{if $\beta=0, q=1$} \\
s\left(\frac{1}{1-q}\beta\right)& \text{if $\beta\neq 0, q\neq 1$} \\
\psi(\beta) & \text{if $\beta\neq 0, q=1$}. 
\end{cases}
\]
\end{cor}
\begin{proof}
It is clear that if $x\in I_F$, then $s^{k_1}(x) = \infty$ and so $G(x) = \infty$ as a result. Thus, from now on we will assume that $x\in D_F$ and we think of $D_F$ as a fixed subset of the $\Psi$-set of $\Gamma$. Next we will take an elementary extension $(\Gamma',\psi)$ of $(\Gamma,\psi)$ with an element $\gamma\in \Psi_{\Gamma'}$ such that $\gamma>\Psi$. Now, if we take the table from Theorem~\ref{psisfunctions}, but we replace $\beta$ with $\beta+\gamma$ and have $x$ range over $D_F\subseteq\Gamma$, then we get the following table, computed in $(\Gamma',\psi)$:
\begin{center}
\begin{tabular}{|c|l|}
\hline
$\beta$ & $G(x) = \psi(\sum_{j=1}^nq_js^{k_j}(x)-\beta-\gamma)\quad\text{(assuming $x\in D_F\subseteq\Gamma$)}$ \\
\hline
if $\psi(\beta+\gamma)>s0$ & $G(x) = 
\begin{cases}
s0 & \text{if $q\neq 0$} \\
s^{k_1+1}(x) & \text{if $q=0$ and $s^{k_1+1}(x)<\psi(\beta+\gamma)$ } \\
G(s^{-k_1-1}(\psi(\beta+\gamma))) & \text{if $q=0$ and $s^{k_1+1}(x) = \psi(\beta+\gamma)$} \\
\psi(\beta+\gamma) & \text{if $q=0$ and $s^{k_1+1}(x) > \psi(\beta+\gamma)$}
\end{cases}$\\
\hline
if $\psi(\beta+\gamma)=s0$ & $G(x) = 
\begin{cases}
G(s^{-k_1}s0) & \text{if $s^{k_1}(x) = s0$ } \\
s0 & \text{if $s^{k_1}(x)>s0$ and $q=0$ } \\
s^{k_1+1}(x) & \text{if $s^{k_1}(x)>s0$, $q\neq 0$, and $s^{k_1+1}(x)<s(q^{-1}(\beta+\gamma))$} \\
G(s^{-k_1-1}s(q^{-1}(\beta+\gamma))) & \text{if $s^{k_1}(x)>s0$, $q\neq 0$, and $s^{k_1+1}(x)=s(q^{-1}(\beta+\gamma))$} \\
s(q^{-1}(\beta+\gamma)) & \text{if $s^{k_1}(x)>s0$, $q\neq 0$, and $s^{k_1+1}(x)>s(q^{-1}(\beta+\gamma))$} \\
\end{cases}$\\
\hline
\end{tabular}
\end{center}
Since we are assuming that $x\in D_F\subseteq\Gamma$, we can apply Corollary~\ref{spsigap} to replace $\psi(\sum_{j=1}^nq_js^{k_j}(x)-\beta-\gamma)$ with $s(\sum_{j=1}^nq_js^{k_j}(x)-\beta)$ and also $\psi(\beta+\gamma) = \psi(-\beta-\gamma)$ with $s(-\beta)$. Finally, we set $\gamma_0:= s(q^{-1}(\beta+\gamma))$ when $q\neq 0$. This gives us the desired table:
\begin{center}
\begin{tabular}{|c|l|}
\hline
$\beta$ & $G(x) = s(\sum_{j=1}^nq_js^{k_j}(x)-\beta)\quad\text{(assuming $x\in D_F\subseteq\Gamma$)}$ \\
\hline
if $s(-\beta)>s0$ & $G(x) = 
\begin{cases}
s0 & \text{if $q\neq 0$} \\
s^{k_1+1}(x) & \text{if $q=0$ and $s^{k_1+1}(x)<s(-\beta)$ } \\
G(s^{-k_1-1}(s(-\beta))) & \text{if $q=0$ and $s^{k_1+1}(x) = s(-\beta)$} \\
s(-\beta) & \text{if $q=0$ and $s^{k_1+1}(x) > s(-\beta)$}
\end{cases}$\\
\hline
if $s(-\beta)=s0$ & $G(x) = 
\begin{cases}
G(s^{-k_1}s0) & \text{if $s^{k_1}(x) = s0$ } \\
s0 & \text{if $s^{k_1}(x)>s0$ and $q=0$ } \\
s^{k_1+1}(x) & \text{if $s^{k_1}(x)>s0$, $q\neq 0$, and $s^{k_1+1}(x)<\gamma_0$} \\
G(s^{-k_1-1}\gamma_0) & \text{if $s^{k_1}(x)>s0$, $q\neq 0$, and $s^{k_1+1}(x)=\gamma_0$} \\
\gamma_0 & \text{if $s^{k_1}(x)>s0$, $q\neq 0$, and $s^{k_1+1}(x)>\gamma_0$} \\
\end{cases}$\\
\hline
\end{tabular}
\end{center}
However we are not done yet; currently $\gamma_0$ is still an external parameter. We will show (or arrange) that $\gamma_0\in \Psi_{\infty}$ (and give an explicit formula for it), which will then yield the corollary. First, we assume that $\beta=0$. If $q\neq 1$, then $\gamma_0 = s(q^{-1}\gamma) = s0$ by Corollary~\ref{sgencoeffsumcor}. If $q=1$, then $\gamma_0=s(\gamma)>\Psi$ and so $\gamma_0\not\in \Gamma$. However, in this case, $s^{k_1+1}(x) \star\gamma_0$ iff $s^{k_1+1}(x)\star \infty$ for $\star\in\{<,=,>\}$ because $s^{k_1+1}(x)\in\Psi$ and both $s(\gamma)$ and $\infty$ are $>\Psi$. Thus we redefine $\gamma_0:=\infty$ if $\beta=0$ and $q=1$. Now we assume that $\beta\neq 0$ and we take yet another elementary extension $(\Gamma'',\psi)$ of $(\Gamma',\psi)$ with an element $\tilde{\gamma}\in \Psi_{\Gamma''}$ such that $\tilde{\gamma}>\Psi_{\Gamma'}$. If $q=1$, then we have
\[
\gamma_0 = s(\beta+\gamma) = \psi(\beta+\gamma-\tilde{\gamma}) = \psi(\beta)
\]
by Fact~\ref{valuationfact} and Lemma~\ref{succid} because $\psi(\gamma-\tilde{\gamma}) = s\gamma>\psi(\beta)\in\Psi$. Otherwise, assume that $q\neq 1$ (and thus $q^{-1}-1\neq 0$). Then we can multiply on the inside by $(q^{-1}-1)^{-1}$ to compute
\[
\gamma_0 = s(q^{-1}(\beta+\gamma)) = \psi(q^{-1}(\beta+\gamma)-\tilde{\gamma}) = \psi(q^{-1}\beta+q^{-1}\gamma-\tilde{\gamma}) = \psi\left(\frac{1}{1-q}\beta+ \frac{q}{1-q}(q^{-1}\gamma-\tilde{\gamma})\right).
\]
Next note that $s((1-q)^{-1}\beta)\in\Psi$ whereas $s(\frac{q}{1-q}(q^{-1}\gamma-\tilde{\gamma})) = s\gamma>\Psi$ by Corollary~\ref{sgencoeffsumcor}. Thus $\gamma_0 = s((1-q)^{-1}\beta)$ by Lemma~\ref{succid}.
\end{proof}

\medskip\noindent
Theorem~\ref{psisfunctions} and Corollary~\ref{ssfunctions} are the heart of the proof of Theorem~\ref{psitogamma}. To round things out, we need to make a few more minor observations before proceeding with our proof of Theorem~\ref{psitogamma}.

\begin{lemma}
$\Psi$ is a linearly independent subset of $\Gamma$ as a vector space over $\Q$.
\end{lemma}
\begin{proof}
Let $\alpha = \sum_{j=1}^nq_j\alpha_j$, $\alpha_1<\cdots<\alpha_n\in\Psi$ and $q_1,\ldots,q_n\in\Q^{\neq}$. By Lemma~\ref{gencoeffsumlemma}, either $\psi(\alpha) = s0$, or $\psi(\alpha) = s\alpha_1$, and so $\alpha\neq 0$.
\end{proof}

%\begin{cor}
%The restriction of an $s$-function $F(x) = \sum_{j=1}^nq_js^{k_j}(x)-\beta$ to $D_F$ is injective.
%\end{cor}

\medskip\noindent
Lemma~\ref{sfunctionstopsi} below describes the values of an $s$-function in the set $\Psi$.

\begin{lemma}
\label{sfunctionstopsi}
Let an $s$-function $F(x) = \sum_{j=1}^nq_js^{k_j}(x)-\beta$ be given and let $F^*$ be its restriction to $D_F$. Then exactly one of the following is true:
\begin{enumerate}
\item $\image F^*\subseteq\Psi$, $\beta = 0$, $n=1$ and $q_1 = 1$,
\item $|(\image F^*)\cap\Psi| = 2$,
\item $|(\image F^*)\cap\Psi| = 1$,
\item $|(\image F^*)\cap\Psi| = 0$.
\end{enumerate}
\end{lemma}
\begin{proof}
If $\beta\not\in\SPAN_{\Q}\Psi\subseteq\Gamma$, then $(\image F^*)\cap\Psi = \emptyset$. Thus assume for the rest of the proof that 
\[
\beta = q_1'\alpha_1+\cdots+q_m'\alpha_m
\]
 for $\alpha_1<\cdots<\alpha_m\in\Psi$ and $q_1',\ldots,q_m'\in\Q^{\neq}$. 
 
 The idea is that we are interested in which values of $x$ will put the expression
 \[
 \underbrace{q_1s^{k_1}(x)+\cdots+q_ns^{k_n}(x)}_{\alpha(x)} - \underbrace{(q_1'\alpha_1+\cdots+q_m'\alpha_m)}_{\beta}
 \]
 into the set $\Psi$. By the $\Q$-linear independence of $\Psi$, it is necessary that nearly all of the components of $\alpha(x)$ and $\beta$ will cancel. We will do this by a case distinction.
 
  If $m>n+1$ or $m<n-1$, then for all $x\in D_F$, the value of $F^*(x)$ will be a linear combination of two or more elements of $\Psi$ with nonzero coefficients so $(\image F^*)\cap\Psi = \emptyset$. Thus further assume that $n-1\leq m\leq n+1$. If $m=0$ (so $\beta=0$), then $F^*(x)\in\Psi$ iff $n=1$ and $q_1 = 1$, by the linear independence of $\Psi$. So further assume that $m>0$. Now we look at three subcases:

(Case 1: $m=n-1$, $m>0$) In this case we can expand out $F^*(x)$ as follows:
\[
F^*(x) = q_1s^{k_1}(x)+\cdots+q_ns^{k_n}(x) - q_1'\alpha_1-\cdots-q_{n-1}'\alpha_{n-1}.
\]
In order for $F^*(x)$ above to be an element of $\Psi$, it is necessary that either $s^{k_1}(x) = \alpha_1$ or $s^{k_n}(x) = \alpha_n$, otherwise the value of $F(x)$ will be a linear combination of two or more elements of $\Psi$. Thus $|(\image F^*)\cap\Psi|\leq 2$ in this case.

(Case 2: $m=n+1$, $m>0$) This case is similar to Case 1 and $|(\image F^*)\cap \Psi|\leq 2$.

(Case 3: $m=n$) We can expand $F^*(x)$ as follows:
\[
F^*(x) = q_1s^{k_1}(x)+\cdots+q_ns^{k_n}(x) - q_1'\alpha_1-\cdots-q_{n}'\alpha_{n}.
\]
In order for $F^*(x)\in\Psi$, it is necessary that $s^{k_j}(x) = \alpha_j$ for $j=1,\ldots,n$. Otherwise the value of $F^*(x)$ will be a linear combination of two or more elements of $\Psi$. Thus $|(\image F^*)\cap\Psi|\leq 1$ in this case.
\end{proof}

\begin{lemma}
\label{psitogammaterms}
Let $t(x):\Gamma_{\infty}\to\Gamma_{\infty}$ be an $L$-term and let $F:\Psi\to\Gamma_{\infty}$ be the restriction $t|\Psi$ of $t$ to $\Psi$. Then there is an increasing sequence $s0=\alpha_0<\alpha_1<\cdots<\alpha_{n-1}<\alpha_n=\infty$ in $\Psi_{\infty}$ such that for $k=0,\ldots,n-1$, the restriction of $F$ to $[\alpha_k,\alpha_{k+1})_{\Psi}$ is given by an $s$-function.
\end{lemma}
\begin{proof}
We do this by induction on the complexity of the $L$-terms.

(Easy Cases) By definition the constant term $\beta$ for $\beta\in\Gamma_{\infty}$ is an $s$-function, and its clear that the set of $s$-functions is closed under $+,-$ and $\delta_n$ for $n\geq 1$.

($\psi$ Case) Let $F(x) = \sum_{j=1}^nq_js^{k_j}(x)-\beta$ be an $s$-function. Then we can determine the value of $\psi(F(x))$ from Theorem~\ref{psisfunctions}. Note that whenever the expression $s^l(x)<\delta$ is not vacuous in the table of Theorem~\ref{psisfunctions}, then it is equivalent to $x<s^{-l}\delta$ (similarly for $=$ and $>$).

($s$ Case) This is similar to the $\psi$ case, except we use Corollary~\ref{ssfunctions}.

($p$ Case) Let $F(x) = \sum_{j=1}^nq_js^{k_j}(x)-\beta$ be an $s$-function. By Lemma~\ref{sfunctionstopsi}, if $\beta=0, n=1,q_1=1$, then $F^*$ is of the form $s^k(x)$ and so
\[
p(F(x)) = \begin{cases}
\infty & \text{if $x\in I_F$} \\
\infty & \text{if $x = \min D_F$ and $k\leq0$} \\
s^{k-1}(x) & \text{if $x>\min D_F$ or $k>0$}
\end{cases}
\] 
Otherwise, $F(x)\in\Psi$ for $0,1$ or $2$ values of $x$, so $p(F(x)) = \infty$ for all $x\in\Psi$ with at most $0$, $1$ or $2$ exceptions.
\end{proof}

\medskip\noindent
We say that a set $I\subseteq\Psi$ is an \textbf{interval in $\Psi$} if there are $\alpha,\beta\in\Psi_{\infty}$ with $\alpha<\beta$ such that $I = [\alpha,\beta)_{\Psi}$. The following is immediate from Theorem~\ref{qe}, and Lemmas~\ref{sfunctionsmonotone} and~\ref{psitogammaterms}:

\begin{cor}
\label{ominimal}
Every definable $A\subseteq\Psi$ is a finite union of intervals in $\Psi$ and singletons.
\end{cor}

\begin{proof}[Proof of Theorem~\ref{psitogamma}]
It follows from quantifier elimination and the fact that $T_{\log}$ has a universal axiomatization that there are $L$-terms $t_1(x),\ldots,t_{n}(x)$ such that on $\Psi$ we have $F(x) = t_k(x)$, for some $i\in \{1,\ldots,k\}$. By Corollary~\ref{ominimal}, the set
\[
D_i:= \{x\in\Psi: F(x) = t_k(x)\}\subseteq \Psi
\]
will be a finite union of intervals and singletons. Furthermore, by Lemma~\ref{psitogammaterms}, the restriction of $F(x)$ to $D_i$ will be given piecewise by $s$-functions in the desired way.
\end{proof}

\begin{cor}[Characterization of definable functions $\Psi\to\Psi$]
\label{psitopsi}
Let $F:\Psi\to\Psi$ be definable in $(\Gamma,\psi)$. Then there is an increasing sequence $s0=\alpha_0<\alpha_1<\cdots<\alpha_{n-1}<\alpha_n=\infty$ in $\Psi_{\infty}$ such that for $k=1,\ldots,n$, the restriction of $F$ to $[\alpha_{k-1},\alpha_{k})_{\Psi}$ is either constant or of the form $x\mapsto s^l(x)$ for some $l\in\Z$.
\end{cor}
\begin{proof}
This follows from Theorem~\ref{psitogamma} and  Lemma~\ref{sfunctionstopsi}.
\end{proof}

\section{Definable subsets of $\Psi$}
\label{stableembeddedness}

\noindent
\emph{In this section we assume that $(\Gamma,\psi)$ is a model of $T_{\log}$}. 

\medskip\noindent
It is clear from Corollary~\ref{ominimal} that each nonempty definable $A\subseteq\Psi$ has a least element. This gives us definable Skolem functions for definable subsets of $\Psi^n$ (see, for example,~\cite[p. 94]{tametopology}).

\medskip\noindent
The following Theorem~\ref{psintopsi} follows immediately from Corollary~\ref{ominimal} and the main result of~\cite{discreteominimal}. For the reader's convenience we supply a more direct and self-contained proof. It is a variant of~\cite[Lemma 4.7]{densepairs}, which itself is a variant of~\cite[Lemma 1]{stronglyminimal}.

\begin{thm}
\label{psintopsi}
Let $n\geq 1$ and suppose that $f:\Psi^{n}\to\Psi$ is a definable function. Then $f$ is definable in the structure $(\Psi;<)$.
\end{thm}
\begin{proof}
We can arrange that $(\Gamma,\psi)$ is $\aleph_0$-saturated.

The case $n=1$ follows from Corollary~\ref{psitopsi}.

Let $n>1$. Let $A$ be the finite set of parameters from $\Gamma$ used to define $f$. For each $a\in\Psi$ we can define the function $f_a:x\mapsto f(a,x):\Psi^{n-1}\to\Psi$. By induction, $f_a$ is definable in the structure $(\Psi;<)$, so we have $c_a\in\Psi^{N_a}$ and a set $\Phi_a\subseteq \Psi^{N_a+(n-1)+1}$ definable in $(\Psi;<)$ such that $\Phi_a(c_a) = \Gamma(f_a)$. We can arrange that $\Phi_a$ is the graph of a function $F_a:\Psi^{N_a+(n-1)}\to\Psi$ such that $F_a(c_a,x) = f_a(x)$ for all $x\in\Psi$. Next let $\Delta_a\subseteq\Psi$ be the $A$-definable set of all $b\in\Psi$ such that the function $f_b:\Psi^{n-1}\to\Psi$ occurs as a section of $F_a$. Note that $a\in\Delta_a$ since $F_a(c_a,x) = f(a,x)$. Thus
\[
\Psi = \bigcup_{a\in\Psi} \Delta_a
\]
By saturation there are $a_1,\ldots,a_k\in\Psi$ such that:
\[
\Psi = \bigcup_{j=1}^k\Delta_j
\]
where $\Delta_j:=\Delta_{a_j}$ for $j=1,\ldots,k$. Let $F_j:=F_{a_j}$, $\Phi_j:=\Phi_{a_j}$, $c_{j}:=c_{a_j}$ and $N_j:=N_{a_j}$ for $j=1,\ldots,k$ and let $N = \max_{1\leq j\leq k}N_j$. Extend each function $F_j:\Psi^{N_j+(n-1)}\to\Psi$ to a function $F_j':\Psi^{N+(n-1)}\to\Psi$ by setting
\[
F_j'(w_1,\ldots,w_N,x):=F_j(w_1,\ldots,w_{N_j},x)\;\; \text{for all $(w_1,\ldots,w_N,x)\in\Psi^{N+(n-1)}$}
\]
so the last $N-N_j$ variables before $x$ are dummy variables. Next define a function $F:\Psi^{1+N+(n-1)}\to\Psi$ by
\[
F(v,w_1,\ldots,w_N,x) = \begin{cases}
F_1'(w_1,\ldots,w_N,x) & \text{if $v=s0$} \\
F_2'(w_1,\ldots,w_N,x) & \text{if $v=s^20$} \\
&\vdots  \\
F_{k-1}'(w_1,\ldots,w_N,x) & \text{if $v=s^{k-1}0$}\\
F_k'(w_1,\ldots,w_N,x) & \text{if $v\geq s^k0$} \\
\end{cases}
\]
Finally, we note the following:
\begin{eqnarray*}
\Psi = \textstyle\bigcup_{j=1}^n\Delta_j&\Rightarrow&\text{for every $a\in\Psi$ there is $j\in\{1,\ldots,k\}$ such that $a\in\Delta_j$} \\
&\Rightarrow& \text{for every $a\in\Psi$ there is $j\in\{1,\ldots,k\}$ and $c\in\Psi^{N_j}$} \\ &&\;\quad\quad\text{such that $f(a,x) = F_j(c,x)$ for every $x\in\Psi$} \\
&\Rightarrow& \text{for every $a\in\Psi$ there is $j\in\{1,\ldots,k\}$ and $c\in\Psi^{N}$} \\ &&\;\quad\quad\text{such that $f(a,x) = F_j'(c,x)$ for every $x\in\Psi$} \\
&\Rightarrow& \text{for every $a\in\Psi$ there is $v\in\Psi$ and $c\in\Psi^N$} \\ && \;\quad\quad\text{such that $f(a,x) = F(v,c,x)$ for every $x\in\Psi$} \\
&\Rightarrow& \text{for every $a\in\Psi$ there is $c\in\Psi^{1+N}$} \\ && \;\quad\quad\text{such that $f(a,x) = F(c,x)$ for every $x\in\Psi$} \\
&\Rightarrow& \forall a\in\Psi \;\exists c\in\Psi^{1+N} \;\forall x\in\Psi (f(a,x) = F(c,x))
\end{eqnarray*}
By definability of Skolem functions, there is a definable function $c=(c_0,\ldots,c_N):\Psi\to\Psi^{1+N}$ such that
\[
\forall a\in\Psi\;\forall x\in\Psi \;(f(a,x) = F(c(a),x))
\]
From the base case of this lemma, we may assume that $c_i:\Psi\to\Psi$ is definable in $(\Psi;<)$ for $i=0,\ldots,N$. Thus $f(z,x):\Psi^n\to\Psi$ agrees with the function $F(c(z),x):\Psi^n\to\Psi$, which is definable in $(\Psi;<)$. This concludes the proof of the induction step.
\end{proof}

\begin{cor}
The subset $\Psi$ of $\Gamma$ is stably embedded in $(\Gamma,\psi)$.
\end{cor}

\section{Final Remarks}

\noindent
\emph{In this section we assume that $(\Gamma,\psi)$ is a model of $T_{\log}$}. In contrast to the o-minimality of $\Psi$ (Corollary~\ref{ominimal}), it is important to note that $(\Gamma,\psi)$ is not even weakly o-minimal because the definable set $\Psi\subseteq\Gamma$ is infinite and discrete. In fact, $(\Gamma,\psi)$ is not even \emph{locally o-minimal} (in the sense of~\cite{localominimality}) because the definable set $(\Psi-\Psi)^{>0}\subseteq\Gamma$ does not have the local o-minimality property at $0$.

\medskip\noindent However, $(\Gamma,\psi)$ is ``o-minimal at infinity'' in the following sense:

\begin{lemma}
If $X\subseteq\Gamma$ is definable in $(\Gamma,\psi)$, then there is $a\in\Gamma$ such that $(a,\infty)\subseteq X$ or $(a,\infty)\cap X = \emptyset$.
\end{lemma}
\noindent
This is immediate from the following claim:
\begin{claim}
Let $F:\Gamma\to\Gamma_{\infty}$ be a definable function. Then there is some $a\in\Gamma$ such that on the restriction $(a,\infty)$, $F$ is either constant, or of the form $x\mapsto qx+\beta$ for $q\in\Q^{\neq}$ and $\beta\in\Gamma$.
\end{claim}
\begin{proof}
By quantifier elimination and universal axiomatization of $T_{\log}$, it follows that $F$ is given piecewise by $L$-terms. In particular, there is some $a\in\Gamma$ such that $F$ is equal to an $L$-term $t(x)$ on $(a,\infty)$. Thus we can prove this by induction on the complexity of $t(x)$.

The cases $t(x) = \beta$ for some $\beta\in\Gamma_{\infty}$ is clear since this is already a constant function. The cases $t(x) = t_1(x)+t_2(x)$, $t(x) = -t_1(x)$, $t(x) = \delta_nt_1(x)$ are also clear. 

If $t(x)$ is constant on $(b,\infty)$ for some $b\in\Gamma$, then so are $\psi(t(x)), s(t(x))$ and $p(t(x))$. If $t(x)$ is $qx+\beta$ on $(b,\infty)$, then $t(x)$ is either strictly increasing and cofinal in $\Gamma$, or strictly decreasing and coinitial in $\Gamma$. Thus $\psi(t(x))$ and $s(t(x))$ will eventually be the constant value $s0$ and $p(t(x))$ will eventually be the constant value $\infty$.
\end{proof}

\medskip\noindent
We conclude with a list of unresolved issues and things left to do:

\begin{enumerate}[(1)]
\item Describe more explicitly the subsets of $\Gamma$ that are definable in $(\Gamma,\psi)$.
\item Describe all definable functions $\Gamma\to\Gamma_{\infty}$.
\item Describe all possible simple extensions $\Gamma\preceq\Gamma\langle c\rangle$.
\item Does $T_{\log}$ have NIP (the Non-Independence Property)?
\item Is $T_{\log}$ \emph{distal}? Distal theories form a subclass of NIP theories which in some sense are purely unstable. See~\cite{distal} for a definition of distality.
\item Is $(\Gamma,\psi)$ \emph{quasi-weakly-o-minimal}, i.e., any definable subset is a finite boolean combination of convex sets and $0$-definable sets? For more information on this property in the o-minimal setting, see~\cite{quasiominimal}.
\item Is $(\Gamma,\psi)$ \emph{d-minimal}, i.e., any definable subset of $\Gamma$ is a union of an open set and finitely many discrete sets? See~\cite[\S3.4]{dminimal} for a discussion of d-minimality in the context of expansions of the real field.
\end{enumerate}

\noindent
Items (3) and (4) are addressed in a future paper,~\cite{gehretNIP}.

\section*{Acknowledgements}

\noindent
This material is based upon work supported by the National Science Foundation under Grant No. 0932078000 while the author was in residence at the Mathematical Sciences Research Institute in Berkeley, California, during the Spring 2014 semester. The author would like to thank the referee for the very careful reading of the manuscript and many helpful suggestions, and would also like to thank Lou van den Dries for his guidance and numerous discussions around the topics of this paper.

\bibliographystyle{amsalpha}	
\bibliography{refs}

\end{document}